\documentclass[12pt,reqno]{amsart}
\usepackage{cmap}
\usepackage[T1]{fontenc}
\usepackage[utf8]{inputenc}
\usepackage[english]{babel}
\usepackage{amsfonts,amssymb,amsthm,amsmath}
\usepackage{url}
\usepackage{enumitem}

\usepackage{secdot}

\usepackage{graphicx}
\usepackage{epstopdf}
\usepackage{a4wide}
\usepackage{xcolor}
\usepackage[colorlinks=true,linktocpage,pdfpagelabels,
bookmarksnumbered,bookmarksopen]{hyperref}
\definecolor{ForestGreen}{rgb}{0.1,0.6,0.05}
\definecolor{EgyptBlue}{rgb}{0.063,0.1,0.6}
\hypersetup{
	colorlinks=true,
	linkcolor=EgyptBlue,         
	citecolor=ForestGreen,
	urlcolor=olive
}

\usepackage[hyperpageref]{backref}

\newtheorem{theorem}{Theorem}
\newtheorem{proposition}[theorem]{Proposition}
\newtheorem{lemma}[theorem]{Lemma}
\newtheorem{corollary}[theorem]{Corollary}

\theoremstyle{definition}
\newtheorem{definition}[theorem]{Definition}
\newtheorem{remark}[theorem]{Remark}

\let\OLDthebibliography\thebibliography
\renewcommand\thebibliography[1]{
	\OLDthebibliography{#1}
	\setlength{\parskip}{1pt}
	\setlength{\itemsep}{1pt plus 0.3ex}
}

\numberwithin{equation}{section}
\numberwithin{theorem}{section}
\numberwithin{equation}{section}
\numberwithin{theorem}{section}

\newcommand{\B}{{B{(x_0)}}}

\makeatletter
\def\namedlabel#1#2{\begingroup
    #2%
    \def\@currentlabel{#2}%
    \phantomsection\label{#1}\endgroup
}
\makeatother

\title[Large solutions of quasilinear elliptic equations]{Large solutions of degenerate and/or singular quasilinear elliptic equations in a ball}

\author[R.~N.~Dhara]{Raj Narayan Dhara}
\address{
Department of Mathematical Analysis and
Applications of Mathematics\\
Faculty of Science, Palacký University\\
17.listopadu 12,
771 46 Olomouc, Czech Republic\\
	\newline\indent
Department of Mathematics and Statistics, Faculty of Science
	\newline\indent
	Masaryk University, Kotl\'a\v{r}sk\'a 2, 611 37 Brno, Czech Republic
}
\email{dhara@math.muni.cz, rajnarayan.dhara@upol.cz}

\date{}

\subjclass[2010]{35J70, 35J92, 35J25, 35B44, 35B51, 35C99}
\keywords{large solutions, boundary blow-up behaviors, degenerate PDE, weighted p-Laplacian}

\thanks{
	RND acknowledges the support of the Czech Science Foundation, project GJ19-14413Y and the support of faculty project IGA\_{}PrF\_{}2022\_{}008.
}

\begin{document}
\begin{abstract} 
	We consider local weak large solutions with its blow-up rate near the boundary to certain class of degenerate and/or singular quasilinear elliptic equation\\ 
${\rm div}(d^{\alpha}(x,\partial{}B)\Phi_p(\nabla u)) = b(x)f(u)$	
 in a ball B, where $f$ is normalized regularly varying at infinity with index $\sigma+1>p-1,\ p>1$. In particular, how the asymptotic behavior of the solution changes over the varying index and degeneracy and/ or singularity present in the equation. We also include the second order blow-up rate for the corresponding semilinear problem.
\end{abstract} 
\maketitle 

\section{Introduction and main result}\label{sec:intro}
Consider the following quasilinear equation
\begin{align}
\Delta_{d^{\alpha},p} u:={\rm div}(d^{\alpha}(x,\partial\Omega)\Phi_p(\nabla u)) = b(x)f(u)\ \text{in}\ \B,\label{eq:1}\\
u(x)\rightarrow\infty\ \ \text{as}\ d(x)\to 0,\label{eq:2}
\end{align}
where $\B:=B_{1}(x_0)$, where $B_{r}(x_0)\subset\mathbb{R}^{N},\, N>1$ is a ball centered at $x_0$ with radius $r>0$,  $d(x,\partial \B):={\rm dist}(x,\partial\B)$ is the distance from $x$ to $\partial\B$ the boundary of $\B$, $-1<\alpha<p-1, p>1$, $\Phi_p(\cdot)=|\cdot|^{p-2}(\cdot)$, $0< b\in C(\B)$ and $f\ge 0$ is a continuously differentialble nondecreasing function. Eventually, one can choose any bounded ball of finite radius but for simplicity we investigate the problem in a unit ball $\B$.  The boundary behavior~\eqref{eq:2} is being understood in the pointwise limit sense, i.e., for any real number $M$, there exists a $\delta>0$ such that $u(x)>M$ whenever $d(x)<\delta$. 
 A positive solution of~\eqref{eq:1} satisfying~\eqref{eq:2} is called a {\em large solution}. The most interesting part for having such large solutions for nonlinear partial differential equations (PDE) satisfying 
comparison principle in some sense is that it provide an upper bound for any
solution of the PDE on a given bounded domain irrespective of boundary conditions.


 We are interested to study the existence and uniqueness of solutions of~\eqref{eq:1}-\eqref{eq:2} and its blow-up rate with general nonlinearity imposed by regularly varying function (see e.g.~\cite{bingham87,sene}), and their assymptotic behavior near the boundary. The various aspects of considered problem may arise in many branches of mathematics and applied mathematics. Therefore, the existence, uniqueness of large solutions and its boundary behavior for the nondegenerate ($\alpha=0,\ p=2,\ b\equiv 1$) case has been discussed and extended by many authors as well. For instance, the problem
 \begin{align}\label{eq:1:clscl}
 \begin{cases} 
 \Delta u=f(u),\quad \Omega,\\
 u\to\infty,\quad x\to\partial\Omega,
 \end{cases}
 \end{align}
where $\Omega$ is a bounded domain and many generalizations of its type have been extensively studied. For instance, Bieberbach (\cite{Bieberbach1916}, 1916) firstly showed that~\eqref{eq:1:clscl}, coming from differential geometry application admits a unique solution when $\Omega$ is a smooth planar domain and $f(u)=e^{u}$. Later on, Rademacher extended Bieberbach's results to three dimensional domain of consideration as an application to mathematical physics.

 Finally, Keller~\cite{Keller57} and Osserman~\cite{osserman57} first came up independently with the necessary and sufficient condition, popularly known as Keller-Osserman condition, or superlinear growth of the nonlinearity $f(u)$ as in~\eqref{assum:f1} for the existence of the large solutions to~\eqref{eq:1:clscl} in a smooth bounded domain $\Omega\subset\mathbb{R}^{N}$. The Keller-Osserman condition essentially implies {\em a priori} estimate in terms of radial solution. Furthermore, as a consequence, it has been shown that it is equivalent to the fact that the nonexistence of entire solution of $\Delta u=f(u)\ \text{in}\ \mathbb{R}^{N}$. It is worth to mention that the problem of existence and uniqueness of type~\eqref{eq:1:clscl} for $f(u)=u^{(N+2)/(N-2)},\ N>2$ was considered by Loewner and Nirenberg~\cite[1974]{nirenberg74}. They were looking for a geometric question when a Riemannian metric will be a complete metric and invariant under M\"obius transformations. In a trend of Bieberbach's type problem, various aspects of it has been treated in a series of papers by several reknown mathematicians. For the sake of our considered problem we discuss some of the preliminary results dealing with $p$-Laplacian. Diaz and Letelier~\cite{diaz93}, first studied the problem~\eqref{eq:1}-\eqref{eq:2}, for $\alpha=0,\ b\equiv\lambda\in\mathbb{R},\ f(u)=u^{\rho},\rho>p-1$ and later on by Matero~\cite{matero96}. Mohammed~\cite{mohammed04} consider the problem~\eqref{eq:1}-\eqref{eq:2} with more general assumption when $b$ is allowed to be unbounded near $\partial\Omega$.

 To get the blow-up rate of~\eqref{eq:1}-\eqref{eq:2}, we follow the approach of Karamata's regularly varying function firstly introduced, by C\^{\i}rstea and R\v{a}dulescu in a series of papers~\cite{radu2002,radu2006,radu2007}. 
We say a measurable function $f:(0,+\infty)\to (0,+\infty)$, is called {\em regularly varying at infinity} with index $\rho\in\mathbb{R}$, written as $f\in RV_{\rho}$, if for $\xi>0$
\begin{align*}
\lim_{u \to\infty} f(\xi u)/f(u) = \xi^{\rho} .
\end{align*} 
We refer to Section~\ref{sec:reg:func} for other information on regularly varying function. The various results using the regularly varying function theory applied to obtain the exact boundary behavior can be seen in~\cite{mohm2007,zhang2017}.
  Furthermore, the exact second order blow-up rate can be seen in~\cite{radu2003}.

In this article we propose a substantially adapted technical analysis of the fact that deos not require different analysis around the {\em degenerate (and/ or singular)} and  {\em regular} points. In contrast to the standard theory, we will establish the existence of a unique large solution for the degenerate and/or singular  $p$-Laplacian elliptic equations. We also provide the exact first and second order blow-up rates near the boundary.

It is noted that the boundary behavior~\eqref{eq:2} has not been imposed a priori, rather we show that under the adequate condition on growth of the nonlinearity $f$ in~\eqref{eq:1}, namely, Keller-Osserman condition gives a unique singular nature near the boundary $\partial\B$ by a uniform rate of  explosion depending on $\alpha, b$ and $f$. Since explosive solutions are not defined on the boundary, we must adapt an adequate version of comparison principle $\bar{u}\le \overline{u}$ in $\B$ in a weak sense for any two local solutions $\bar{u},\, \overline{u}$ to~\eqref{eq:1}, see Theorem~\ref{thm:comp:1}.

Furthermore, we have the following assumptions for $f$ and the coefficient $b$ are as follows.
\begin{description}
\item[\namedlabel{itm:f1}{(F)}] Assume that
$f$ satisfies
\begin{align}
C^{1}([0,\infty))\ni f (>0)\ \text{is nondecreasing on}\ (0,\infty),\ f(0)=0;\label{assum:f1}\tag{F1}\\
 \int_{1}^{\infty}F^{-1/p}(t)\, dt<\infty,\ \text{where}\ F(t):=\int_{0}^{t}f(s)\,ds;\label{assum:f2}\tag{F2}
\end{align}
\item[\namedlabel{itm:b1}{(B)}] $b\in C(\B)$, is positive in $\B$;
\item[\namedlabel{itm:b2}{($\tilde{B}$)}] $b$ satisfies the condition that if there exists $x_0\in\B$ such that $b(x_0)=0$, then there exists an annulus $B_{R}(x_0)\setminus \overline{B_{r}(x_0)}=:A_{x_0}(r,R)\Subset\B,\ 0<r<R<1$ with the same center $x_0$ such that $b(x)>0$ for all $x\in\partial B_{x_0}(R)$;
\item[\namedlabel{assum:B2}{(B$^\prime$)}] Boundary assymptotic of $b$ is given as for positive constants $b_i, (i=1,2)$ such that
\begin{align*}
b_1:=\liminf_{d(x)\to 0}\frac{b(x)}{d^{\alpha-\frac{\alpha p}{2}}(x)k^{p}(d(x))}\le \limsup_{d(x)\to 0}\frac{b(x)}{d^{\alpha-\frac{\alpha p}{2}}(x)k^{p}(d(x))}=:b_2 ;
\end{align*}
In particular, when $b_1=b_2=c$, for some constant $c>0$, we can assume
\begin{align}\label{assum:b1}
b(x)=c d^{\alpha-\frac{\alpha p}{2}}(x)k^p\left(d(x)\right)+o\left(d^{\alpha-\frac{\alpha p}{2}}(x)k^p(d(x))\right)\ \text{as}\ d(x)\to 0,\tag{B1}
\end{align}
\end{description}
where $k\in\mathcal{K}$, the set of all positive, monotonic functions $k\in L^{1}(0,\nu)\cap C^{1}(0,\nu)$ satisfying
\begin{align}\label{eq:cls:k}
\int_{0}^{t} s^{-\alpha/2}k(s)\, ds=:K(t),\ \text{and}\ \lim_{t\to 0+}\left(\dfrac{K(t)}{t^{-\alpha/2}k(t)}\right)^{(i)}=l_i,\ i=\{0,1\}.
\end{align}
The class of functions $\mathcal{K}$ was introduced by C\^{i}rstea and R\u{a}dulescu~\cite[p-449]{radu2002} and Mohammed~\cite[p-484]{mohm2007} for nondecreasing and nonincreasing functions respectively. 

\begin{remark}\label{rem:k:lim}
It is also noted that
\begin{align*}
l_0=0,\ \text{for any}\ k\in\mathcal{K},\quad \lim_{t\to 0+}\dfrac{K(t)(t^{-\alpha/2}k(t))'}{t^{-\alpha}k^2(t)}=1-l_1.
\end{align*}
Furthermore, we note that $0\le l_1\le 1$ if $k$ is nondecreasing and $l_1\ge 1$ if $k$ is nonincreasing.
\end{remark}
\begin{remark}
One may think of~\eqref{assum:b1} as the following possible cases compatible with our analysis here.
\begin{itemize}
\item[Case 1.] For $\alpha-\frac{\alpha p}{2}\ge 0$ while $k>0$ and a slowly varying function.
\item[Case 2.] For $\alpha-\frac{\alpha p}{2}< 0$ while $k(t)>0$ and faster growth at zero than $t^{\alpha-\frac{\alpha p}{2}}$, i.e., $k\in NRVZ_{l_1^{-1}-1+\alpha/2}$, see Definition~\ref{def:nrvz}.
\end{itemize}
\end{remark}
We shall use the following definitions of solution, sub- and supersolution in  Sobolev space $W^{1,p}_{\rm loc}(\B)$.
\begin{definition}\label{def:sub:sol}
A function $u\in W^{1,p}_{\rm loc}(\B)\cap C(\B)$ is a positive (weak) {\em supersolution} (respectively, {\em subsolution}) of~\eqref{eq:1}
if
\begin{align}\label{eq:5}
\int_{\B}\Phi_p(\nabla u)\nabla \phi\, d^{\alpha}dx  + \int_{\B} b(x)f(u) \phi\, dx \ge (\text{respectively,}\ \le)\ 0,
\end{align}
for all $0\le \phi \in W^{1,p}_{c}(\B)$.

A function $u\in W^{1,p}_{\rm loc}(\B)\cap C(\B)$ is a (weak) {\em solution }of~\eqref{eq:1} if the equality of~\eqref{eq:5} holds for all $\phi \in W^{1,p}_{c}(\B)$. The large solution should be interpreted as follows:\\
For every positive integer $k$ we have $k-u\le 0$ on $\partial\B$ in the sense that $(k-u)^{+}\in W^{1,p}_{0}(\B)$.
\end{definition}
\begin{remark}
The above integrals in~\eqref{eq:5} is well-defined by assumptions~\eqref{assum:b1}-\eqref{assum:f1}.
\end{remark}

 We stated the main results in the following theorems.

\begin{theorem}[Existence of blow-up solutions]\label{thm:exis:blw}
Let $\B\subset\mathbb{R}^{N}$ be bounded domain and $p>1$. Suppose that $f$ satifies~\ref{itm:f2} and $b\in C(\overline{\B})$ satisfies~\ref{itm:b1}. Then~\eqref{eq:1}-\eqref{eq:2} admits a nonnegative solution $u\in W^{1,p}_{\rm loc}(\B)\cap C^{1,\beta}(\B),\ \beta\in(0,1)$.
\end{theorem}

\begin{theorem}[First order blow-up rate]\label{thm:main:gen}
Let $f\in RV_{\sigma+1}$ with $\sigma>p-2$. Furthermore, 
assume that the assumption~\ref{assum:B2} for $b$ 
on $\partial\B$.
Then any local weak solution $u_{\sigma}$ of
\begin{align}\label{eq:gen:1}
\begin{cases}
\Delta_{d^{\alpha},p} u={\rm div}(d^{\alpha}\Phi_{p}(\nabla u)) = b(x)f(u)\qquad \text{in}\ \B,\\
u(x)\to \infty\ \ \text{as}\ x\to \partial\B,
\end{cases}
\end{align}
satisfies
\begin{align}\label{eq:growth1}
\xi_{2}\le \liminf_{d(x)\to 0}\frac{u(x)}{\phi(K(d(x)))}\le \limsup_{d(x)\to 0}\frac{u(x)}{\phi(K(d(x)))}\le \xi_{1},
\end{align}
where 
\begin{align}
\xi_{1}:=\left(\dfrac{2+l_1(1-\alpha)(2+\sigma-p)}{b_1(2+\sigma)}\right)^{1/(2+\sigma -p)},\quad \xi_{2}:=\left(\dfrac{2+l_1(1-\alpha)(2+\sigma-p)}{b_2(2+\sigma)}\right)^{1/(2+\sigma -p)}
\end{align}
and $\phi$ is defined by
\begin{align}\label{eq:inv}
\int_{\phi(t)}^{\infty}\dfrac{ds}{(qF(s))^{1/p}}= t>0.
\end{align}
In particular, when $b_1=b_2=c$ in~\ref{assum:B2}, then $u_{\sigma}$ satisfies
\begin{align}\label{eq:growth}
\lim_{d(x)\to 0}\dfrac{u_{\sigma}(x)}{\phi(K(d(x)))}=\xi_0,\ \text{where}\ \xi_{0}:=\left(\dfrac{2+l_1(1-\alpha)(2+\sigma-p)}{c(2+\sigma)}\right)^{1/(2+\sigma -p)},
\end{align}
\end{theorem}
The uniqueness of solutions of~\eqref{eq:1} can be provided under some additional assumption of $f$ as follows.
\begin{align}\label{eq:cond:f:unique}
f\ge 0,\quad f(t)/t^{p-1}\ \text{is nondecreasing on}\ (0,\infty).
\end{align}
It is worth to note that if $f\in RV_{\sigma +1}$ satisfies condition~\eqref{eq:cond:f:unique}, then $\sigma>p-2$.
\begin{theorem}[Uniqueness of large solutions]\label{thm:unique}
Let $f\in RV_{\sigma+1}$ with $\sigma>p-2,\ p>1+\alpha$ satisfy \ref{itm:f1} and~\eqref{eq:cond:f:unique}. Furthermore, 
assume that the assumption~\ref{assum:B2} for $b$ 
on $\partial\B$. Then~\eqref{eq:1} admits a unique weak solution in $W^{1,p}_{\rm loc}(\B)\cap C^{1,\alpha}(\B)$.
\end{theorem}
The second order blow-up rate required the following additional assumptions on the class of functions $\mathcal{K}$.
We shall denote for $0\le l_1\le 1$, $\mathcal{K}_{l_1}=\mathcal{K}_{0}\cup\mathcal{K}_{(0,1]}$, where
\begin{align*}
\mathcal{K}_{0}:=\{k\in\mathcal{K}:\ l_1 = 0 \},\quad
\mathcal{K}_{(0,1]}:=\{k\in\mathcal{K}:\ l_1 = (0,1] \}.
\end{align*}
Some basic examples of $k\in\mathcal{K}_{l_1}$ are
$(i)$ $k(t)=t^q,\ q>\alpha/2,\ l_1 =(1+q-\alpha/2)^{-1}$;
 $(ii)$ $k(t)=\ln(1+ t^q),\ q>1+\alpha/2,\ l_1 =(1+q-\alpha/2)^{-1}$;
 $(iii)$ $k(t)=\exp(t^q)-1,\ q>1+\alpha/2,\ l_1 =(1+q-\alpha/2)^{-1}$.
 
For the purpose of getting second order behavior of the large solution, we want an additional assumption for the class $\mathcal{K}_{l_1}$, namely, for $\tau,\zeta>0$, we define
\begin{align*}
\mathcal{K}_{[0,1],y}:=&
\left\{k\in\mathcal{K}_{l_1}:\ \lim\limits_{t\to 0+}\frac{1}{y(t)}\left( \left(\dfrac{K(t)}{t^{-\alpha/2}k(t)}\right)^{(1)}-l_1\right)=e_k\in \mathbb{R}\right\};\\
\text{in particular,}&\\
\mathcal{K}_{(0,1],\zeta}:=&
\left\{k\in\mathcal{K}_{(0,1]}:\ \lim\limits_{t\to 0+}\frac{1}{t^{\zeta}}\left( \left(\dfrac{K(t)}{t^{-\alpha/2}k(t)}\right)^{(1)}-l_1\right)=E_k\in \mathbb{R}\right\};\\
\mathcal{K}_{(0,1],\tau}:=&
\left\{k\in\mathcal{K}_{(0,1]}:\ \lim\limits_{t\to 0+}\frac{1}{(-\ln t)^{-\tau}}\left( \left(\dfrac{K(t)}{t^{-\alpha/2}k(t)}\right)^{(1)}-l_1\right)=L^{\#}\in \mathbb{R}\right\};\\
\mathcal{K}_{0,\tau}:=&
\left\{k\in\mathcal{K}_{0}:\ \lim\limits_{t\to 0+}\frac{1}{(-\ln t)^{-\tau}} \left(\dfrac{K(t)}{t^{-\alpha/2}k(t)}\right)^{(1)}=L^{*}\in \mathbb{R}\right\};\\
\mathcal{K}_{0,\zeta}:=&
\left\{k\in\mathcal{K}_{0}:\ \lim\limits_{t\to 0+}\frac{1}{t^{\zeta}} \left(\dfrac{K(t)}{t^{-\alpha/2}k(t)}\right)^{(1)}=L_{*}\in \mathbb{R}\right\} .
\end{align*}
The second order behaviors for the semilinear problem read as follows.
\begin{theorem}\label{thm:p:2}
Let $f\in RV_{\sigma+1}$ with $\sigma>0$. Furthermore, 
assume that $b$ satisfying~\eqref{assum:b2} 
nearby $\partial\B$:
\begin{equation}\label{assum:b2}
b(x)=k^2(d(x))\left(1+B_{0} d^{\theta}(x)+o(d^{\theta}(x)\right),\ \text{for some constant}\ B_{0},\theta>0\ \text{and}\ k\in\mathcal{K}_{0,\tau}.\tag{B}
\end{equation}
Then the second order blow up rate estimate of any local weak solution $u_{\sigma}$ of
\begin{align}\label{eq:gen:2}
\begin{cases}
\Delta_{d^{\alpha},2} u:={\rm div}(d^{\alpha}\nabla u) = b(x)f(u)\qquad \text{in}\ \B,\\
u(x)\to \infty\ \ \text{as}\ x\to \partial\B,
\end{cases}
\end{align}
is given by
\begin{align}\label{eq:growth:2nd:a}
u_{\sigma}(x)=\xi_0\phi(K(d(x)))\left[1+\chi (-\ln d(x))^{-\tau} +o((-\ln d(x))^{-\tau}) \right],
\end{align}
where
\begin{align}
\xi_0:=\left(\frac{2}{2+\sigma} \right)^{\frac{1}{\sigma}},\quad \chi:=\dfrac{(1-\alpha)\sigma L^{*}}{(p-1)(3+\sigma)},
\end{align}
and $\phi$ be uniquely determined by
\begin{align}\label{eq:inv:2}
\int_{\phi(t)}^{\infty}\dfrac{ds}{(2F(s))^{1/2}}= t>0.
\end{align}
\end{theorem}
More generally, we can formulate with the following assumption.
\begin{description}
\item[Hypothesis 1] We assume, for a proper choice of second order approximation $y$ in Theorem~\ref{thm:p:2a}, that the limit
\begin{align}\label{lim:y:assump}
\lim_{r\to 0+}\frac{r^{\theta}}{y(r)}=:G(\theta,y)
\end{align}
exists.
\end{description}
\begin{remark}
The different choices of $y(r)$ gives the different limits in~\eqref{lim:y:assump} as follows.
\begin{align*}
G(\theta,y)=
\begin{cases}
1,\ y(r)=r, \theta=1;\\
0,\ y(r)=r, \theta>1;\\
0,\ y(r)=(-\ln r)^{-\tau},\ \tau,\theta>0;\\
H(\zeta-\theta),\ y(r)=r^{\omega},\ \zeta,\theta>0,\ \omega:=\min\{\zeta,\theta \};\\
\end{cases}
\end{align*}
\end{remark}
\begin{theorem}\label{thm:p:2a}
Let $f\in RV_{\sigma+1}$ with $\sigma>0,\ 2>1+\alpha$ and \ref{itm:f1} hold. Furthermore, 
 assume that $b$ satisfying~\eqref{assum:b2} 
nearby $\partial\B$ for $k\in\mathcal{K}_{(0,1],\zeta},\ \zeta>0$.
Then the second order blow up rate estimate of any local weak solution $u_{\sigma}$ of~\eqref{eq:gen:2} is given by
\begin{align}\label{eq:growth:2nd:b}
u_{\sigma}(x)=\xi_0\phi(K(d(x)))\left(1+\chi y(d(x)) +o(y(d(x))) \right),
\end{align}
where
\begin{align}\label{eq:xi0}
\xi_0:=\left(\frac{2+l_1(1-\alpha)\sigma}{2+\sigma} \right)^{\frac{1}{\sigma}},\quad \chi:=\dfrac{\sigma((1-\alpha/2)e_k-l_1)-B_0(2+\sigma l_1)G(\theta,y)}{\sigma(3+\sigma+l_1)-(\alpha/2)\sigma^2 l_1^2+\alpha\sigma l_1},
\end{align}
$y(t)\in C([0,\nu))$ is a nondecreasing function such that $y(0)=0$ and $t/y(t)\rightarrow 0$ as $t\to 0+$ and $\phi$ be uniquely determined by~\eqref{eq:inv:2}.
\end{theorem}
\begin{remark}
We can choose $y(t)=(-\ln t)^{-\tau}$ in Theorem~\ref{thm:p:2} and the proof will follow the same line with the different limits from the hypothesis to be adapted in Lemma~\ref{lem:lims}.
\end{remark}

We organize the article as follows. In Section~\ref{sec:preli}, we define  weighted Sobolev spaces which arise naturally to study such degenerate problems. In this settings, we define the solutions, sub- and supersolutions to the considered problem and give an adapted comparison principle. It also includes the existence of solution to~\eqref{eq:1} in between sub- and supersolutions which explode on the boundary. In Section~\ref{sec:exis:blw}, using the aforesaid setup and recalling some standard regularity results, we eventually show the existence of the blow-up solutions to problem~\eqref{eq:1}-\eqref{eq:2} in Theorem~\ref{thm:exis:blw}.
In Section~\ref{sec:1d}, we give a direct one dimensional calculation to finding the growth rate of the large solution to~\eqref{eq:1} for the case $f(u)=u^q$. The rate has been derived in terms of parameters depending on the weights involved, decay rate of $b(x)$ towards the boundary, nonlinearity, namely, $q>p-1>\alpha$. In Section~\ref{sec:reg:func}, we recall some of the results from the theory of Karamata's regularly varying functions. These tools essentially conform existence and uniqueness of solution to~\eqref{eq:1} with its first order blow-up rate involving the regularly varying index in Section~\ref{sec:blow:1st} and second order blow-up rate in Section~\ref{sec:2nd:blw}.

\section{Solution, subsolution and supersolution}\label{sec:preli}
We investigate the problem of type~\eqref{eq:1}, where the degeneracy appears at the boundary $\partial\B$ for $\alpha>0$ and/ or around the set $\{x\in\B:\ \nabla u(x)=0 \}$, i.e., the elliptic operator ${\rm div}(d^{\alpha}\Phi_{p}(\nabla\cdot))$ has a nonnegative characterisitc form around those prescribed areas. We say the problem~\eqref{eq:1} is singular for $\alpha<0$. In order to  see  how the degeneracy, appearing for the power of the distance to the boundary $\partial\B$, effects the blow up rate of the large solution to~\eqref{eq:1} at the boundary, we investigate the problem~\eqref{eq:1} in weighted Sobolev space. In particular, we shall adapt the classical comparison principles to weak solutions properly defined in weighted Sobolev spaces, see Definition~\ref{def:sub:sol}. Therefore, we start with the definitions of weighted Sobolev spaces and their properties needed in the present context of the article.
\subsection{Weighted Sobolev spaces and their properties}
Let $0\le w\in L^{1}_{\rm loc}(\B)$ be a locally integrable nonnegative function in $\B$. Then the Radon measure $\mu$ canonically identified with the weight $w(x)$ is defined as $\mu(E)=\int_{E}w(x)\, dx$. Therefore, $d\mu(x)=w(x)dx$, where $dx$ is the $n$-dimensional Lebesgue measure. We say the weight function $w$ is $p$-admissible, see~\cite[Section 1.1]{heino12}. We choose for our particular interest $w(x)=d^{\alpha}(x),\ \alpha>-1$ for which it is a $p$-admissible weight, see~\cite[p-10]{heino12}. Note that, by definition and along with the doubling property, $\mu$ and Lebesgue measure $dx$ are mutually absolutely continuous, i.e., the almost everywhere (a.e.) expressions do not need to mention with respect to which measure separately. We use the notation $u^{+}:=\max(u,0),\ u^{-}:=\max(-u,0)$ and $1/p + 1/p'=1,\ 1<p<\infty$. It is noted that $u=u^{+}-u^{-}$ and $|u|=u^{+}+u^{-}$. Before we proceed to the definitions of weak solutions, sub- and supersolutions to~\eqref{eq:1}, we give a brief introduction to weighted Sobolev spaces.


\begin{definition}{\cite[Section 1.9]{heino12}}\label{def:space}
The weighted Sobolev space $W^{1,p}(\B;d^{\alpha})\ \left(W^{1,p}_{0}(\B;d^{\alpha})\right)$ is defined as the completion of the set
\begin{align*}
\left\{\phi\in C^{\infty}(\B)\ \left(\phi\in C^{\infty}_{0}(\B)\right):\quad \int_{\B}|\phi|^p\, d\mu + \int_{\B}|\nabla\phi|^p\, d\mu=:\|\phi\|_{1,p}^{p}<\infty\right\}.
\end{align*}
\end{definition}
It means that a function $u$ lies in the class $W^{1,p}(\B;d^{\alpha})$ if and only if $u\in L^{p}(\B;d^{\alpha})$ and there exists a vector valued function $z\in L^{p}(\B;d^{\alpha};\mathbb{R}^N)$ such that
\begin{align*}
\int_{\B}|\phi_i-u|^p\, d\mu \stackrel{k\to 0}{\longrightarrow} 0\quad\text{and}\quad \int_{\B}|\nabla\phi_i-z|^p\, d\mu \stackrel{k\to 0}{\longrightarrow} 0,
\end{align*}
for some sequence $\phi_i\in C^{\infty}(\B)$. The function $z$ is called the gradient of $u$ in $W^{1,p}(\B;d^{\alpha})$ and is denoted by $z=\nabla u$. Moreover, $\nabla u$ is uniquely defined function in $L^{p}(\B;d^{\alpha})$. It is noted that $W^{1,p}(\B;d^{\alpha})$ and $W^{1,p}_{0}(\B;d^{\alpha})$ are reflexive Banach spaces under the norm $\|\cdot\|_{1,p}$,~see~\cite[Section 1.9]{heino12}. We note that $W^{1,p}(\Omega;d^{\alpha})$ and $W^{1,p}_{0}(\Omega;d^{\alpha})$ are also defined in a similar way on an open bounded domain $\Omega$.

We say a function $u\in W^{1,p}_{\rm loc}(\B;d^{\alpha})$ if and only if $u\in W^{1,p}(\B';d^{\alpha})$ for each open set $\B'$ whose closure is compact subset of $\B$. A useful lemma below needs to be recalled from~\cite{heino12}.
\begin{lemma}{\cite[Lemma 1.17]{heino12}}\label{lem:aux:1}
If $u\in  W^{1,p}_{0}(\B;d^{\alpha})$ satisfying $\nabla u=0$, then $u=0$.
\end{lemma}

In general, we can deal with such weight functions $w$ which either may vanish somewhere in $\overline{\B}$ and/or increase to infinity. In present context, we consider the weight function $w(x):=d^{\alpha}(x)$. Then, for $-1<\alpha<p-1,$ $d^{\alpha}\in A_p$ or in Muckenhoupt class, see~\cite[Chapter 15]{heino12}. 
Note that, when $d\mu(x)=d^{\alpha}(x)dx$, $W^{1,p}_{\rm loc}(\B;d^{\alpha})=W^{1,p}_{\rm loc}(\B)$. We denote the class of functions from $W^{1,p}_{0}(\B)$ with compact support in $\B$ as $W^{1,p}_{c}(\B)$.

The following comparison principle for weak solutions to degenerate quasilinear equations is needed to conclude our final claim with large solution. The proof can be adapted from~\cite[Lemma 3.18]{heino12}. 
\begin{lemma}\label{thm:comp:1}
If $\bar{v}\in W^{1,p}_{\rm loc}(\B)\cap C(\B)$ be a weak supersolution and $\underline{v}\in W^{1,p}_{\rm loc}(\B)\cap C(\B)$ be a weak subsolution to~\eqref{eq:1} in the sense of Definition~\ref{def:sub:sol}. Then the inequality $(\bar{v}-\underline{v})^{-}\in W^{1,p}_{0}(\B)$ implies $\underline{v}\le \bar{v}$ a.e. in $\B$.
\end{lemma}
\begin{proof}[Proof of Lemma~\ref{thm:comp:1}]
Let us choose $w:=(\bar{v}-\underline{v})^{-}\in W^{1,p}_{0}(\B)$, and $w\ge 0$. It follows then
\begin{align*}
0\ge \int_{\B} \left(\Phi_p(\nabla\underline{v})-\Phi_p(\nabla \bar{v})\right)\nabla w\, d^{\alpha}dx +\int_{\B} b(x)f(\underline{v}) w\, dx - \int_{\B} b(x)f(\bar{v}) w\, dx~~~~~~~~~~~~~~~~~~~~~\\
= \int_{\B}\frac{|\nabla\underline{v}|^{p-2}+|\nabla \bar{v}|^{p-2}}{2} |\nabla(\bar{v}-\underline{v})^-|^2\, d^{\alpha}dx~~~~~~~~~~~~~~~~~~~~~~~~~~~~~~~~~~~~~~~~~~~~~~~\\
 +\int_{\{\bar{v}<\underline{v}\}}\dfrac{|\nabla \underline{v}|^{p-2}-|\nabla \bar{v}|^{p-2}}{2} (|\nabla \underline{v}|^2 -|\nabla \bar{v}|^2)\, d^{\alpha}dx + \int_{\{\bar{v}<\underline{v}\}}b(x)(f(\underline{v}) - f(\bar{v})) (\bar{v}-\underline{v})^{-} \, dx\ge 0.
\end{align*}

Indeed, since $b\ge 0$ and $f$ is nondecreasing, every summand in this last expression is nonnegative, and hence we
obtain that $w = 0$ a.e. in $\B$, {\em i.e.}, $\underline{v}\le \bar{v}$ a.e. in $\B$. 
\end{proof}

\section{The existence theorem}\label{sec:exis:blw}
The following result guarantees that the a priori existence of a {\em large solution} in between large sub- and supersolution to the problem~\eqref{eq:1}.
\begin{theorem}\label{thm:exis}
Let $\underline{u},\overline{u}\in W^{1,p}_{\rm loc}(\B)\cap C(\B)$ are the weak sub- and supersolution of~\eqref{eq:1} according to Definition~\ref{def:sub:sol} in $\B$ such that
\begin{align*}
\underline{u}(x)\le \overline{u}(x)\qquad\text{a.e. in}\ \B.
\end{align*}
Then \eqref{eq:1}
and with further conditions 
\begin{align}
\lim_{d(x)\to 0}\underline{u}(x)=\infty\qquad\text{and}\qquad\lim_{d(x)\to 0}\overline{u}(x)=\infty,
\end{align}
 \eqref{eq:1}-\eqref{eq:2} possesses a (weak) solution $u\in W^{1,p}_{\rm loc}(\B)\cap C(\B)$ in between $\underline{u}$ and $\overline{u}$.
\end{theorem}
\begin{proof}
For each $n\ge 1$, we consider
\begin{align*}
B_n:= B(x_0,1-\frac{1}{n}).
\end{align*}
Clearly, $B_n\subseteq B_{n+1}$ for $n\ge 1$ and we can choose $n\ge n_0$ large enough such that $\partial B_n$ is of class $C^2$ and
\begin{align*}
\partial B_n=\left\{ x\in \B:\ {\rm dist}(x,\partial\B)=\dfrac{1}{n}\right\}\subset\B.
\end{align*}
Note that, $d^{\alpha}(x)={\rm dist}^{\alpha}(x,\partial\B)$ is bounded in $ B_n$ for each $n\ge 1$. Following~\cite[Main Theorem, p-51]{hess74}, considering (H1), (H2) there, $A_{1}(0,0,\nabla u)=d^{\alpha}(x)\Phi_{p}(\nabla u)$, $p(x,u(x),{\bf 0})= - b(x)f(u)$, in particular, the problem
\begin{align}\label{eq:10}
\begin{cases}
-{\rm div}(d^{\alpha}\Phi_{p}(\nabla u)) + b(x)f(u) = 0\qquad &\text{in}\  B_n,\\
u=(\underline{u}+ \overline{u})/2\quad \text{on}\ \partial B_n,
\end{cases}
\end{align}
possesses a solution $u_n\in W^{1,p}( B_n)$ such that $u_n - (\underline{u}+ \overline{u})/2\in W^{1,p}_{0}( B_n)$ and
\begin{align}\label{eq:11}
\underline{u}_{|_{ B_n}}\le u_n\le \overline{u}_{|_{ B_n}}\qquad\ \text{in}\  B_n.
\end{align}
Taking advantage of~\eqref{eq:11} and the $C^{1,\beta} (0<\beta<1)$ interior regularity theory by DiBenedetto~\cite{dibenedetto83}-Tolksdorf~\cite{tolksdorff84} together with Ascoli-Arzel\`a theorem gives rise to an existence of subsequence of $\{u_n\}_{n\ge 1}$, say $\{u_{n_m}\}_{m\ge 1}$, for which
\begin{align*}
\lim_{m\to \infty}\|u_{n_m}-u_0\|_{C^{1,\gamma}(\overline{B}_{n_0})}=0,\quad 0\le\gamma<\beta<1,
\end{align*}
for some solution $u_0\in C^{1,\beta}( B_{n_0})$ of~\eqref{eq:10}. Now consider the new sequence $\{{u_{n_m}}{|_{ B_{n_1}}}\}_{m\ge 1}$. The previous argument also shows the existence of a subsequence of $\{{u_{n_m}}{|_{ B_{n_1}}}\}_{m\ge 1}$, relabeled by $n_m$, such that, for some $u_1\in C(\overline{B}_{n_1})$,
\begin{align*}
\lim_{m\to \infty}\|u_{n_m}-u_1\|_{C^{1,\gamma}(\overline{\B}_{n_1})}=0\quad 0\le\gamma<\beta<1.
\end{align*}
Note that, the continuity of the solution $u_1$ in $\overline{B}_{n_1}$ gives us $u_{1}|_{ B_{n_0}}=u_0$. Repeating this argument infinitely, the pointwise limit of the diagonal sequence $\{ u_{m_m}\}$ gives us the required solution in between $\underline{u}$ and $\overline{u}$.
\end{proof}
Next we formulate the existence theorem which follows the technique originally given by Keller~\cite{Keller57} applied for the case $
p=2$. To proceed with that, we need the following regularity result for weak solutions to quasilinear equations due to DiBenedetto~\cite{dibenedetto83} and Tolksdorff~\cite{tolksdorff84}.
\begin{theorem}[$C^{1,\beta}$ interior regularity~{\cite[DiBenedetto]{dibenedetto83}}]\label{thm:reg:loc}
Let $\Omega\subset\mathbb{R}^{N}$ be bounded domain and $p>1$. Assuming~\ref{itm:f1}-\ref{itm:b1}, suppose that $u\in W^{1,p}_{\rm loc}(\Omega)\cap L^{\infty}_{\rm loc}(\Omega)$ be a weak solution of~\eqref{eq:1}. Then for a given compact subset $K\Subset \Omega$, there is $\beta\in(0,1)$ and a possitive constant $C$, depending on $N,p,\|b\|_{\infty}, \|u\|_{\infty}$ and $K$ such that
\begin{align}\label{eq:reg:est}
|\nabla u(x)|\le C\quad\text{and}\quad|\nabla u(x)-\nabla u(y)|\le C|x-y|^{\beta},\ x,y\in K.
\end{align}
\end{theorem}

\begin{theorem}[Existence of a BVP]\label{thm:1eaux}
Let $\Omega\subset\mathbb{R}^{N}$ be a bounded domain and $-1<\alpha<p-1$. Suppose that $f\in RV_{\sigma+1},\ \sigma >p-2$. Then for every given function $\eta\in W^{1,p}(\Omega;d^{\alpha})$ 
such that $F\circ \eta\in L^{1}(\Omega)$, the problem
\begin{align}\label{eq:1eaux}
\begin{cases}
{\rm div}(d^{\alpha}\Phi_p(\nabla u)) = f(u),\ x\in\Omega,\\
u-\eta\in W^{1,p}_{0}(\Omega;d^{\alpha}).
\end{cases}
\end{align}
admits a weak solution $u\in W^{1,p}(\Omega;d^{\alpha})$.
\end{theorem}
\begin{proof}
A weak (variational) solution of the problem~\eqref{eq:1eaux}, we understand a function $u\in W^{1,p}(\Omega;d^{\alpha})$ which minimizes the Euler functional
\begin{align*}
J(u):=\int_{\Omega}\left[\frac{1}{p}|\nabla u(x)|^{p}d^{\alpha}(x)+F(u(x))\right]\, dx
\end{align*}
on a set $\mathcal{K}:=\{v\in W^{1,p}(\Omega;d^{\alpha}):\   v-\eta\in W^{1,p}_{0}(\Omega;d^{\alpha})\ \text{and}\ F\circ v\in L^{1}(\Omega)\}$. Note that by Theorem~\ref{prop:assymp:reg:func}(\cite[Theorem 2.1, p-53]{sene}), we have $F(u)\ge (\sigma+2)^{-1}u^{\sigma+2}L(u)$ for $\sigma>-2$ for all $u\in \mathbb{R}$. The weight function $d^{\alpha}$ for $-1<\alpha<p-1$ satisfies the weighted Poincar\'e inequality. Hence the proof can follow the same lines of~\cite[Theorem 3.4]{rnd5}.
\end{proof}
One can also find a proof of Theorem~\ref{thm:1eaux} in~\cite[Example 1(ii),p-219]{Drabek1998}.

We say the following condition (weighted Keller-Osserman) be satisfied by $f(u)$:
\begin{description}
\item[\namedlabel{itm:f2}{(\text{$\tilde{F}$})}]~ if it is a single-valued real continuous function defined for all real values of $u$ and if there exist a positive nondecreasing continuous function $h(u)$, a radial weight functions $w>0$ and invertible maps $s,t:\mathbb{R}\rightarrow (0,\infty]$ such that $f(u)> h(u)$ and
\begin{align}\label{eq:ko:cond}
\lim\limits_{t\to\infty}\int_{0}^{t}w(t^{-1}(s))^{\frac{1}{p-1}}\left[\frac{p}{p-1}\int_{0}^{s}w(s^{-1}(z))^{\frac{1}{p-1}} h(z)\, dz\right]^{-1/p}\, ds<\infty.
\end{align}
\end{description}
\begin{theorem}\label{thm:aux:blw:exis}
Let $f$ satisfy the condition~\ref{itm:f2}, $-1<\alpha<p-1$ and let $u$ be a $W^{1,p}_{loc}(B_{R}(x_0))$-solution of ${\rm div}(w(|x-x_0|)\Phi_p(\nabla u(x))) = f(u(x))$
 in a ball $B_{R}(x_0)$ as per Definition~\ref{def:sub:sol}. Then there exists a decreasing function $\phi$ determined by $\tilde{h}$ in~\ref{itm:f2} such that
\begin{align}\label{eq:aux:5}
u(x)\le \phi(d(x,\partial B_{R}(x_0))).
\end{align}
Furthermore, the function $\phi$ has the limits
\begin{align}
\phi(r)\to \infty,\quad r\to 0,\label{eq:aux:6}\\
\phi(r)\to -\infty,\quad r\to \infty.\label{eq:aux:7}
\end{align}
\end{theorem}
As a consequence of Theorem~\ref{thm:aux:blw:exis}, we obtain the following Theorems.
\begin{theorem}
If $f(u)$ satisfies~\ref{itm:f2}, then 
\begin{align}\label{eq:entire}
{\rm div}(w(|x|)\Phi_p(\nabla u(x))) = b(x)f(u(x))
\end{align}
 has no entire solution.
\end{theorem}
\begin{proof}
On the contrary let us suppose that an entire solution for~\eqref{eq:entire} exists. Then $u$ is defined in a sphere of any radius $r$ centered at a point $y$ and by Theorem~\ref{thm:aux:blw:exis}, $u(y)\le \phi(r), \forall\ r$. This essentially implies that $u(y)\le q$ for any real number $q$ which is an absurd. Hence, there cannot have any entire solution $u$ of~\eqref{eq:entire}.
\end{proof}
\begin{theorem}[Existence theorem with unweighted absorption term]\label{thm:exis:1a}
Let $B_{R}(x_0)=:D\subset\mathbb{R}^{N}$, $0<R<\infty$. Assume $-1<\alpha<p-1$ and that $f\in RV_{\sigma+1},\ \sigma >p-2$ satisfies weighted Keller-Osserman condition in~\ref{itm:f2}. Then the problem
\begin{align}
{\rm div}(d(x,\partial D)^{\alpha}\Phi_p(\nabla u(x))) = f(u(x)),\ x\in D,\label{eq:1aux:1a}\\
u(x)\to\infty,\ \ x\to \partial D\label{eq:1aux:1b}
\end{align}
admits a nonnegative solution $u\in W^{1,p}_{\rm loc}(D)\cap C^{1,\beta}(D),\ \beta\in(0,1).$
\end{theorem}
\begin{proof}[Proof of Theorem~\ref{thm:exis:1a}]
The proof can be adapted by the same technique from~{\cite[Theorem III]{Keller57}}.
Note that $d(x,\partial D)^{\alpha}=(R-|x-x_0|)^{\alpha}$ is radial with the respect to the center $x_0$ of the ball $D$.

 We consider for each $k=1,2,\cdots$, $u_k\in W^{1,p}(D;d^{\alpha})$ be the weak solutions of the following boundary value problem (see e.g. Theorem~\ref{thm:1eaux}, \cite[Theorem 4.3, p-250]{diaz1985} for existence),
\begin{align}
{\rm div}(d(x,\partial D)^{\alpha}\Phi_p(\nabla u(x))) = f(u(x)),\ x\in D,\label{eq:1aux:1aa}\\
u(x)=k,\quad	x\in \partial D.\label{eq:1aux:1bb}
\end{align}
Then by comparison principle in Lemma~\ref{thm:comp:1} and Theorem~\ref{thm:reg:loc}, the sequence $\{u_k\}_{k=1}^{\infty}$ is increasing in $k$ and $u_k\in C^{1,\beta}(K)$ for any compact subset $K\subset D$, for some $0<\beta<1$.
Since $f$ satisfies the weighted Keller-Osserman condition, then Theorem~\ref{thm:aux:blw:exis} holds, and for each point $z\in D$,
\begin{align*}
u_k(z)\le \phi(d(z,\partial D)), \quad\text{for all}\ k,
\end{align*}
where $\phi$ is the function appearing in Theorem~\ref{thm:aux:blw:exis}.
Thereby, in every compactly embedded subset $K\Subset D$, as $k\to\infty$ the $u_k$ converges uniformly to a limit, say $u\in W^{1,p}_{\rm loc}(D)$, and it solves weakly ${\rm div}(d(x,\partial D)^{\alpha}\Phi_p(\nabla u(x))) = f(u(x)),\ x\in D$. It is evident by construction of $u_k$'s that $u$ also satisfies $u(x)\to\infty$ as $x\to\partial D$. The said regularity of $u$ inherits from the regularity of $u_k$'s.
\end{proof}
We now show the existence of such $\phi$ in Theorem~\ref{thm:aux:blw:exis}.
\begin{proof}[Proof of Theorem~\ref{thm:aux:blw:exis}]
We shall adapt the proof by the same method as in Keller-Osserman style~\cite[Theorem I]{Keller57}. We shall use the weak comparison principle instead of maximum principle there.
Note that since $u\in W^{1,p}_{\rm loc}(B_{R}(x_0))$, we can always find a concentric ball $B:=B_{\tilde{R}}(x_0),\ 0<\tilde{R}<R$ such that $u$ admits a Sobolev trace on $\partial B$.
Given $u$ is a solution of~\eqref{eq:1} in $B$ and $u$ has a trace on $\partial B$. We define a function $v_{k,\tilde{R}}$ in $B$ and on $\partial B$ as the solution\footnote{Existence of radial solutions~\cite[Section 7]{pucci2006e} and its uniqueness~\cite[Section 8]{pucci2006u}} of
\begin{align}
{\rm div}(w(|x-x_0|)\Phi_{p}(\nabla v(x)))&=h(v(x)),\quad x\in B,\label{eq:aux:1}\\
v&=k,\quad x\in\partial B,\label{eq:aux:2}
\end{align}
where $h$ is the function in the assumption~\ref{itm:f2} and recall that $f(u)\ge h(u)$. We assume that the constant in~\eqref{eq:aux:2} satisfies
\begin{align}
\mathbb{N}\ni k\ge u(x),\quad x\in\partial B.
\end{align}
The existence of the nonhomogeneous boundary value problems~\eqref{eq:aux:1}-\eqref{eq:aux:2} can be obtained by Theorem~\ref{thm:1eaux}.

Then clearly $u$ is {\em subsolution} of~\eqref{eq:aux:1}-\eqref{eq:aux:2}. Then by comparison Lemma~\ref{thm:comp:1}, we have for each $k$,
\begin{align}\label{eq:aux:3}
u(x)\le v_{k,\tilde{R}}(x)\quad a.e.\ \text{in}\ B.
\end{align}
Define the function $\phi(s)$ for $s=\tilde{R}-|x-x_0|=d(x,\partial B)$ by
\begin{align}\label{eq:aux:3a}
\phi(s):=\lim\limits_{\tilde{R}\to R}\lim\limits_{k\to\infty}v_{k,\tilde{R}}(x),\quad	x\in B,
\end{align}
a quantity which could be finite or infinite to be determined.
Since $v_{k,\tilde{R}}$ is increasing function of $k$, we have by definition of limit, for every $k$, $v_{k,\tilde{R}}(s)\le \phi(s)$ for $0\le s\le \tilde{R}$. This together with~\eqref{eq:aux:3} give us
\begin{align}\label{eq:aux:4}
u(x)\le \phi(R-|x-x_0|)=\phi(d(x,\partial B_R(x_0)))\quad \text{for almost all}\ x\in B_R(x_0).
\end{align}
Indeed, if not for almost all $x\in B_R(x_0)$,~\eqref{eq:aux:4} holds, then there exists a point $x\in B_R(x_0)\setminus B$ such that $|B_R(x_0)\setminus B|\neq 0$ and $u(x)> \phi(R-|x-x_0|)=\phi(d(x,\partial B_R(x_0)))$. But then we can always find a concentric ball of radius $\tilde{R}<R$ such that~\eqref{eq:aux:3} and hence by~\eqref{eq:aux:3a},~\eqref{eq:aux:4} hold unless $|B_R(x_0)\setminus B|= 0$. This inequality~\eqref{eq:aux:4} is the desired inequality~\eqref{eq:aux:5}. Furthermore, we will show that $\phi(s)$ is finite for $0\le s <R$, and that $\phi(s)$
satisfies~\eqref{eq:aux:6}-\eqref{eq:aux:7} and is a decreasing function of $s$.


Let $v_k(x)=\psi_k(|x|)=\psi_k(r)$, $r$ denotes the distance from the center of the sphere and computing the radial form of~\eqref{eq:aux:1}-\eqref{eq:aux:2}, we conclude that $\psi_k(r)$ is the solution of
\begin{align}
w(r)\left((p-1)(\psi'(r))^{p-2}\psi''(r)+\frac{w'(r)}{w(r)}\Phi_{p}(\psi'(r))+\frac{N-1}{r}(\psi'(r))^{p-1}\right)~~~~~~~~~~~~~~\nonumber\\
= h(\psi(r)),\quad 0\le r<\tilde{R}\label{eq:rad:1}\\
\psi'(0)=0,\label{eq:rad:2}\\
\psi(\tilde{R})=k.\label{eq:rad:3}
\end{align}
Equation~\eqref{eq:rad:2} is coming from the regularity of $\psi$ at $r=0$. 
Note that since the solution of the problem~\eqref{eq:rad:1}-\eqref{eq:rad:2} is unique, every real no. $k$ uniquely determines $\psi_k(0)$ which is by comparison principle monotonic increasing in $k$. Therefore, $k$ is itself determined by $\psi_k(0)$. Thus we may replace the boundary condition~\eqref{eq:rad:3} by
\begin{align}\label{eq:rad:4}
\psi_k(0)=\psi_{0}.
\end{align}
As $\psi_0$ increases, as $k=\psi(\tilde{R})$ increases. We will show that $\lim\limits_{\tilde{R}\to R}\psi(\tilde{R})$ is infinite for some value of $\psi_0$ which is the $\lim\limits_{\tilde{R}\to R}\lim\limits_{k\to\infty}\psi_{0}$ defined in~\eqref{eq:aux:3a}.

We multiply~\eqref{eq:rad:1} by $r^{N-1}$, and we obtain
\begin{align}\label{eq:rad:7}
\left( r^{N-1}w(r)\psi'(r)^{p-1}\right)'=r^{N-1}h(\psi(r)).
\end{align}
Integrating~\eqref{eq:rad:7} from $0$ to $r$, we have
\begin{align}\label{eq:rad:8}
w(r)\psi'(r)^{p-1}=r^{1-N}\int_{0}^{r} s^{N-1}h(\psi(s))\, ds
\end{align}
Recall that $\psi$ is locally $C^{1,\beta},\ \beta\in(0,1)$, see  Theorem~\ref{thm:reg:loc}. Hence $\psi$ locally solves the problem~\eqref{eq:rad:8}.
We observe from~\eqref{eq:rad:8} that $\psi'\ge 0$. Therefore, $\psi$ is a nondecreasing function and we have from~\eqref{eq:rad:8}
\begin{align}\label{eq:rad:9}
w(r)\psi'(r)^{p-1}\le r^{1-N}h(\psi(r))\int_{0}^{r} s^{N-1}\, ds=\frac{r}{N}h(\psi(r)).
\end{align}
We use~\eqref{eq:rad:9} in~\eqref{eq:rad:1}, and we obtain
\begin{align}\label{eq:rad:9a}
(w(r)\psi'(r)^{p-1})'\ge \left(1-\frac{N-1}{N} \right)h(\psi(r))=\frac{h(\psi(r))}{N}.
\end{align}
Note also that since $w(r)>0$ and $\psi'\ge 0$, then from~\eqref{eq:rad:1} we have
\begin{align}\label{eq:rad:12}
(w(r)\psi'(r)^{p-1})'\le h(\psi(r)).
\end{align}
Combining~\eqref{eq:rad:9a} and~\eqref{eq:rad:12}, we have
\begin{align}\label{eq:rad:13}
h(\psi(r))\ge (w(r)\psi'(r)^{p-1})'\ge \frac{h(\psi(r))}{N}.
\end{align}
Multiply~\eqref{eq:rad:13} by $\frac{p}{p-1}w^{\frac{1}{p-1}}(r)\psi'(r)$, we obtain
\begin{align}
\frac{p}{p-1}w^{\frac{1}{p-1}}(r)h(\psi(r))\psi'(r)\ge (w^{\frac{p}{p-1}}(r)\psi'(r)^{p})'\ge \frac{p}{p-1}\frac{w^{\frac{1}{p-1}}(r)h(\psi(r))\psi'(r)}{N}.
\end{align}
Integrating from $0$ to $r$, we have
\begin{align}\label{eq:rad:14}
H(\psi,\psi_0)\ge w(r)^{\frac{p}{p-1}}\psi'(r)^{p}\ge \frac{H(\psi,\psi_0)}{N},
\end{align}
where $H(\psi,\psi_0)=\frac{p}{p-1}\int_{\psi_0}^{\psi}w^{\frac{1}{p-1}}(\psi^{-1}(z))h(z)\, dz$. Next taking the $p$-th root of the reciprocal of each term~\eqref{eq:rad:14} and integrate again from $0$ to $r$, we obtain
\begin{align}\label{eq:rad:15}
\int_{\psi_0}^{\psi(r)}w(\psi^{-1}(z))^{\frac{1}{p-1}}(H(z,\psi_0))^{-1/p}\, dz
\le r
\le {N}^{1/p}\int_{\psi_0}^{\psi(r)}w(\psi^{-1}(z))^{\frac{1}{p-1}}(H(z,\psi_0))^{-1/p}\, dz.
\end{align}
We can write from~\eqref{eq:rad:3}, \eqref{eq:rad:4} and \eqref{eq:rad:15}
\begin{align*}
\int_{\psi_k(0)}^{k}w(\psi^{-1}(z))^{\frac{1}{p-1}}(H(z,\psi_k(0)))^{-1/p}\, dz
\le \tilde{R}
\le {N}^{1/p}\int_{\psi_k(0)}^{k}w(\psi^{-1}(z))^{\frac{1}{p-1}}(H(z,\psi_k(0)))^{-1/p}\, dz.
\end{align*}
Then we have for any $\psi_{k}(0)\in\mathbb{R}$,
\begin{align*}
\tilde{R}
\le {N}^{1/p}\int_{\psi_k(0)}^{k}w(\psi^{-1}(z))^{\frac{1}{p-1}}(H(z,\psi_k(0)))^{-1/p}\, dz~~~~~~~~~~~~~~~~~~~~~~~~~~~~~\\
\le {N}^{1/p}\int_{\psi_{k}(0)}^{\infty}w(\psi^{-1}(z))^{\frac{1}{p-1}}(H(z,\psi_{k}(0)))^{-1/p}\, dz.
\end{align*}
Hence from~\eqref{eq:rad:15} we can write as
\begin{align}\label{eq:rad:15a}
\int_{\psi_0}^{\psi(r)}w(\psi^{-1}(z))^{\frac{1}{p-1}}(H(z,\psi_0))^{-1/p}\, dz
\le r
\le {N}^{1/p}\int_{\psi_0}^{\infty}w(\psi^{-1}(z))^{\frac{1}{p-1}}(H(z,\psi_0))^{-1/p}\, dz.
\end{align}

In order to converge the integral in~\eqref{eq:rad:15a} as $\psi$ becomes infinite, we have the weighted Keller-Osserman condition~\eqref{eq:ko:cond} as follows
\[
\lim_{t\to\infty}\int_{0}^{t}w(t^{-1}(z))^{\frac{1}{p-1}}(H(z,0))^{-1/p}\, dz<\infty.
\]
Then the integral~\eqref{eq:rad:15a} also converges for any value of $\psi_0$. Suppose we say the limit of the integral as $I(\psi_0)$, then by~\eqref{eq:rad:14} and~\eqref{eq:rad:15a} it follows that
\begin{align}\label{eq:rad:16}
I(\psi_0)\le r\le N^{1/p}I(\psi_0),
\end{align}
where
\begin{align}\label{eq:rad:17}
I(\psi_0)=\lim_{\psi(r)\to\infty}\int_{\psi_0}^{\psi(r)}w(\psi^{-1}(z))^{\frac{1}{p-1}}(H(z,\psi_0))^{-1/p}\, dz.
\end{align}
It is evident that for each $\psi_0$, $\psi$ becomes infinite at a finite value of $r$ in the range obtained in~\eqref{eq:rad:16}. 

Since the value of $r\in [I(\psi_0), N^{1/p}I(\psi_0)]$ depends on $\psi_0$, we denote the corresponding value of $r$ as $\tilde{R}:[0,\infty]\mapsto (0,\infty)$ by $\psi_0\to \tilde{R}(\psi_0)$, i.e., $r=\tilde{R}(\psi_0)$ for some $\psi_0$. Observe that, by uniqueness and continuous dependence of solutions of~\eqref{eq:rad:1}-\eqref{eq:rad:3}, the function $\tilde{R}(\psi_0)$ is continuous and nonincreasing.

Furthermore, for some $\psi_0$,
\begin{align}\label{eq:rad:17a}
I(\psi_0)\le \tilde{R}(\psi_0)\le N^{1/p}I(\psi_0),
\end{align}
and in~\eqref{eq:rad:17}, the integral 
\begin{align}\label{eq:rad:18}
I(\psi_0)\rightarrow +\infty,\ \text{as}\ \psi_0\to -\infty;\quad
I(\psi_0)\rightarrow 0,\ \text{as}\ \psi_0\to +\infty.
\end{align}
Since $\psi_0\to \tilde{R}(\psi_0)$ is continuous for all $\psi_0$, by~\eqref{eq:rad:17a}, $\psi_0\to \tilde{R}(\psi_0)$ inherits the same limits of~\eqref{eq:rad:18}. 
Therefore, from~\eqref{eq:rad:16}, it is evident that
\begin{align*}
\tilde{R}(\psi_0)\rightarrow +\infty,\ \text{as}\ \psi_0\to -\infty;\quad
\tilde{R}(\psi_0)\rightarrow 0,\ \text{as}\ \psi_0\to +\infty.
\end{align*}
 Now we are ready to define $\phi(s)$ as the {\em inverse} of $\tilde{R}(\psi_0)$ as
\begin{align*}
\phi(s):=\min\{\psi_k(0)|\ \tilde{R}(\psi_k(0))=s \},
\end{align*}
which is decreasing and satisfies~\eqref{eq:aux:6}-\eqref{eq:aux:7}. Consequently, by~\eqref{eq:aux:4}
\[
u(x)\le \phi(R-|x-x_0|)=\phi(d(x,\partial B_{R}(x_0)))\quad \text{for almost all}\ x\in B_R(x_0).
\]
\end{proof}

We give now proof of the existence Theorem~\ref{thm:exis:blw} for the problem~\eqref{eq:1}-\eqref{eq:2}.
\begin{proof}[Proof of Theorem~\ref{thm:exis:blw}] Proof follows the steps.\\
{\sc Step 1.}
Let for each $k=1,2\cdots$, $u_k\in W^{1,p}_{\rm loc}(\B)\cap C(\B)$ be a weak solution of
\begin{align}\label{eq:1aux}
\begin{cases}
{\rm div}(d^{\alpha}(x,\partial\B)\Phi_p(\nabla u)) = b(x)f(u),\ x\in\B,\\
u(x)=k,\ \ x\in\partial\B.
\end{cases}
\end{align}
Here the boundary condition can be seen as $(u(x)-k)\to 0$ as $x\to y$ for all $y\in\partial\B$. Then by Theorem~\ref{thm:1eaux} we have the existence of solutions to~\eqref{eq:1aux}.\\
{\sc Step 2.} Note that we have $\Phi_{p}(t)=|t|^{p-1}{\rm sgn\, t}$, and consequently, as zero if $t = 0$. Since $f(0)=0$, it is easily seen that $u\equiv 0$ is a solution of the above Dirichlet problem~\eqref{eq:1aux} with $k=0$. Then by the comparison principle Lemma~\ref{thm:comp:1}, we have
\begin{align*}
0\le u_k(x)\le u_{k+1}(x),\quad x\in\B,\quad\forall\ k=0,1,2,\cdots.
\end{align*}
{\sc Step 3.} 
We show that $\{u_k\}$ is uniformly bounded on compactly embedded sub-domains of $\B$. 
Note that for each $x\in\B$ we can find a unique point on $y\in \partial\B$ such that $d(x,\partial\B)=|x-y|=R-|x-x_0|$.\\
If $b(x)>0,\ x\in B(x_0)$, then there is a concentric compact ball $\bar{B}_{r}(x_0)\subset\B$ of radius $r,\ 0<r<1$ 
such that
 $b(x)>0,\ \forall x\in \bar{B}_{x_0}(r)$. Let $0<m:=\min\{b(x),\ x\in B_{r}(x_0) \}$, it exists since $b\in C(\bar{B}_{r}(x_0))$. Furthermore, let $w$ be a blow-up solution (weak) of
\begin{align*}
\begin{cases}
{\rm div}(d^{\alpha}(x,\partial\B)\Phi_p(\nabla u)) = mf(u),\ x\in B_{r}(x_0),\\
u(x)\to\infty\ \ x\to\partial B_{r}(x_0).
\end{cases}
\end{align*}
Note that $d^{\alpha}(x,\partial\B)\in L^{\infty}(\bar{B}_{r}(x_0))$ be radial with respect to $x_0$.
Then  
 we conclude the existence of such blow-up solution follows from Theorem~\ref{thm:exis:1a}. 

 Now we can see that $u_k$ is subsolution to ${\rm div}(d^{\alpha}(x,\partial\B)\Phi_p(\nabla u)) = mf(u),\ x\in B_{r}(x_0)$ with $u_{k}(x)\le w(x)$ for $x\in \partial B_{r}(x_0)$, for each $k=1,2,\cdots$, and
\begin{align*}
&\int_{B_{r}(x_0)}\Phi_p(\nabla u_k)\nabla\phi\, d^{\alpha}(x,\partial\B)dx+\int_{B_{r}(x_0)}mf(u_k)\phi\, dx\\
&\stackrel{f\ge 0}{\le} \int_{B_{r}(x_0)}\Phi_p(\nabla u_k)\nabla\phi\, d^{\alpha}(x,\partial\B)dx+\int_{B_{r}(x_0)}b(x)f(u_k)\phi\, dx=0,\quad\forall\ \phi(\ge 0)\in W^{1,p}_{0}(B_{r}(x_0)).
\end{align*}
The comparison Lemma~\ref{thm:comp:1}, we have $u_k\le w$ in $B_{r}(x_0)$ for all $k=1,2,\cdots$. Since $w$ is locally bounded we obtain that $u_k\le C$ in $B_{r}(x_0)$ for all $k=1,2,\cdots$ and for some $C>0$.

 Thus, there is a compact ball $\bar{B}_{r}(x_0), 0<r<1$ and a positive constant $C_r$ such that $0\le u_k\le C_r,\ k=1,2,\cdots$ on $B_{x_0}(r)$. Since, $r$ is arbitrary, thereby we conclude that $\{u_k\}$ is uniformly bounded on $\bar{B}_{r}(x_0), 0<r<1$. 
\\
{\sc Step 4.} The DiBenedetto-Tolksdorf $C^{1,\beta}$ interior regularity results implies that the sequences $\{u_k\}$ and $\{\nabla u_k\}$ are equicontinuous in a compact subset $K\Subset\B$. Hence, we can choose a subsequence (Arzel\`a–Ascoli theorem), relabeled again by $\{u_k\}$, such that $u_k\to u$ and $\nabla u_k\to v$ uniformly on compact subsets $K\Subset\B$ for some $u \in C(\B)$ and $v\in (C(\B))^N$. Infact, $v=\nabla u$ on $\B$, and from the interior $C^{1,\beta}$ estimate~\eqref{eq:reg:est} we conclude that $\nabla u\in C^{\beta}(\B),\ \beta\in(0,1)$. Thus $u\in W^{1,p}_{\rm loc}(\B)\cap C^{1,\beta}(\B)$. Let $K\Subset\B$ and $\phi\in W^{1,p}_{0}(\B)$ such that ${\rm supp}\,\phi\subseteq K$. Again from~\eqref{eq:reg:est}, we have that $|\nabla u_k|^{p-1}|\nabla\phi|\le C|\nabla\phi|$, and the continuity of the mapping $\mathbb{R}^{N}\ni\zeta\to|\zeta|^{p-2}\zeta$ implies that $d^{\alpha}(x)|\nabla u_k|^{p-2}\nabla u_k\cdot\nabla\phi\to d^{\alpha}(x)|\nabla u|^{p-2}\nabla u\cdot\nabla\phi$ for $x\in K$ as $k\to\infty$. Therefore, by Lebesgue dominated convergence theorem, we have
\begin{align*}
\int_{K} d^{\alpha}(x)|\nabla u_k|^{p-2}\nabla u_k\cdot\nabla\phi\, d^{\alpha}dx \to \int_{K}d^{\alpha}(x)|\nabla u|^{p-2}\nabla u\cdot\nabla\phi\, d^{\alpha}dx.
\end{align*}
Furthermore, the monotonicity of $f$ as $0\le f(u_k)\le f(u_{k+1})$, and continuity of $f$ as $f(u_k(x))\to f(u(x))$ for each $x\in K$, together with monotone convergence theorem implies that
\begin{align*}
\int_{K} bf(u_k)\phi\, dx\to \int_{K} bf(u)\phi\, dx.
\end{align*}
Therefore, it follows that
\begin{align*}
\int_{K}d^{\alpha}(x)|\nabla u|^{p-2}\nabla u\cdot\nabla\phi\, d^{\alpha}dx=-\int_{K} bf(u)\phi\, dx,\quad \phi\in W^{1,p}_{0}(\B),\ {\rm supp}\,\phi\subseteq K,
\end{align*}
and hence $u$ is a local solution (weak) of ${\rm div}(d^{\alpha}\Phi_p(\nabla u)) = b(x)f(u),\ x\in\B$. Moreover, since $u_k = k$ on $\partial\B$, we conclude that $u(x)\to\infty$ as $x\to\partial\B$.
\end{proof}
\section{The blow-up rate in one dimension}\label{sec:1d}
Let us assume, for simplicity, $b(x)\equiv 1,\ f(u)=u^q,\ q>p-1>\alpha$ in~\eqref{eq:1}.
We want to find the precise blow-up rate at the boundary of~\eqref{eq:1}-\eqref{eq:2}, at least in the radially symmetric case. Thus, we first investigate the blow-up rate of the following one dimensional degenerate problem.
\begin{align}\label{eq:12}
\begin{cases}
((R-x)^{\alpha}\Phi_{p}(u'(x)))'=u^q,\quad x\in(0,R),\\
\lim\limits_{x\to R}u(x)=\infty,\\
\end{cases}
\end{align}
for some $p-1>\alpha\in\mathbb{R}$. Next we employ the following change of variables
\begin{align}
u(x)=(R-x)^{-\beta}\psi(x),\quad x\in(0,R),
\end{align}
where $\beta>0$, to be determined, and successively, we have
\begin{align*}
\Phi_{p}(u'(x))=(R-x)^{-(\beta+1)(p-1)}\Phi_{p}(\tilde{\psi}(x));~~~~~~~~~~~~~~~~~~~~~~~~~~~~~~~~~~~~~~~~~~~~~~~\\
 ((R-x)^{\alpha}\Phi_{p}(u'(x)))'~~~~~~~~~~~~~~~~~~~~~~~~~~~~~~~~~~~~~~~~~~~~~~~~~~~~~~~~~~~~~~~~~~~~~~~~~~~~~~~&\\
 =[(\beta+1)(p-1)-\alpha](R-x)^{-(\beta+1)(p-1)+\alpha-1}\left(\Phi_{p}(\tilde{\psi}(x))+(R-x)\Delta_{p}\tilde{\psi}(x)\right),
\end{align*}
where
\begin{align}\label{eq:psi:til}
\tilde{\psi}(x):=\beta\psi(x)+(R-x)\psi'(x).
\end{align}
Then the equation~\eqref{eq:12} turns out to be
\begin{align}\label{eq:13}
[(\beta+1)(p-1)-\alpha]\left(\Phi_{p}(\tilde{\psi}(x))+(R-x)\Delta_{p}\tilde{\psi}(x)\right)=(R-x)^{-\beta q+(\beta+1)(p-1)-\alpha+1}\psi^{q}(x).
\end{align}
subject to the boundary conditions
\begin{align*}
0<\psi(R)<\infty,
\end{align*}
so that the exact blow-up rate of $u$ at $R$, will be given by $\beta$. 
Since we are looking for the blow-up rate when $x\to R$, we assume
\begin{align*}
\lim\limits_{x\to R}(R-x)\psi'(x)=0,
\end{align*}
which gives from~\eqref{eq:psi:til} that
\begin{align*}
\tilde{\psi}(R)=\beta\psi(R)>0.
\end{align*}
Since $\tilde{\psi}$ is continuous on $[0,R]$, there exists a $\delta>0$ such that $\tilde{\psi}(x)>0,\ x\in(R-\delta,R]$. This helps us to get the following limits using the continuity of the operator $\Phi_p(\cdot)$ for positive functions nearby the point $R$.
\begin{align*}
\lim_{x\to R}\Phi_p(\tilde{\psi}(x))=\beta^{p-1}\psi^{p-1}(R),\quad \lim_{x\to R}(R-x)\Delta_p(\tilde{\psi}(x))=0.
\end{align*}
This necessarily gives us
\begin{align}\label{eq:blw:rate}
\psi(R)=\left[\beta^{p-1}((\beta+1)(p-1)-\alpha)\right]^{\frac{1}{q-(p-1)}}\qquad\text{and}\quad 
\beta=\frac{p-\alpha}{q-(p-1)}.
\end{align}
\begin{remark}
From the above blow-up rate near the end point, it seems that the blow-up is getting demolished for $\alpha\ge p$. However, when $q<p-1$ and $\alpha>p$, we may still have from~\eqref{eq:blw:rate}, the blow-up on the boundary, but in that case we do not meet the conditions in~\ref{itm:f1} unless $q>p-1$
\end{remark}
\begin{remark}
One may also consider
\begin{align}\label{eq:121}
\begin{cases}
((R-x)^{\alpha}\Phi_{p}(u'(x)))'=b(x)u^q,\quad x\in(0,R),\\
\lim\limits_{x\to R}u(x)=\infty,\\
u(0)=0,
\end{cases}
\end{align}
where $b(x)=B(x)(R-x)^{\gamma},\ B(R)\neq 0,\ \gamma\ge 0$. Then the expressions in~\eqref{eq:blw:rate} turn out to
\begin{align*}
\psi(R)=\left[\frac{\beta^{p-1}((\beta+1)(p-1)-\alpha)}{B(R)}\right]^{\frac{1}{q-(p-1)}}\qquad\text{and}\quad 
\beta=\frac{p+\gamma-\alpha}{q-(p-1)}.
\end{align*}
\end{remark}

\section{Nonlinearity as regularly varying functions}\label{sec:reg:func}
We now consider the original problem~\eqref{eq:1}-\eqref{eq:2} with some assumptions on the absorption term $f(u)$ and the coefficient $b(x)$. We assume $f$ belongs to a class of regularly varying functions, defined below. A member of this class of functions to be considered as an absorption term makes the problem more accessible in application point of view.
\begin{definition}\label{def:rv}
A measurable function $f:\mathbb{R}^{+}\to \mathbb{R}^{+}$, is called {\em regularly varying at infinity} with index $\rho\in\mathbb{R}$, written as $f\in RV_{\rho}$, if for $\xi>0$
\begin{align}\label{eq:rv}
\lim_{t \to\infty} f(\xi t)/f(t) = \xi^{\rho} .
\end{align}
Without any obligation we can define a regularly varying function at zero. This is equivalent to say that $f(t)$ is regularly varying at infinity if and only if $f(t^{-1})$ is regularly varying at zero.
\end{definition}
\begin{remark}
When $\rho=0$, we call $f$ is {\em slowly varying}, or $f \in RV_0$. Clearly, if $f\in RV_{\rho}$, then $f(u)/u^{\rho} \in RV_0$. This shows that for a slowly varying function $L(u)$, its always possible to represent a $\rho$-varying function as $u^{\rho}L(u)$. The canonical $\rho$-varying function is $u^{\rho}$. The functions $\log(1+u),\ \log\log(e+u)$, $\exp\{(\log u)^\alpha\},\ 0<\alpha<1$ are slowly varying function. Furthermore, any measurable function with positive limit at infinity is also slowly varying.
\end{remark}
Here we list up some properties of the regularly varying functions, see Seneta~\cite{sene}.
\begin{proposition}{(Uniform Convergence Theorem~\cite[Theorem 1.1, p-2]{sene})}\label{prop:unif:conv}
If $f\in RV_{\rho}$, then~\eqref{eq:rv} holds uniformly for $\xi\in[a,b]$ for every fixed $0<a<b<\infty$.
\end{proposition}
\begin{proposition}{(The Karamata Representation Theorem~\cite[Theorem 1.3.1, p-12]{bingham87} or~\cite[Theorem 1.2, p-2]{sene})}\label{prop:kara:rep}
A function $L$ is slowly varying at infinity if and only if $L$ can be represented as
\begin{align}
L(t)=z(t)\exp\left(\int_{a}^{t}\frac{y(\tau)}{\tau}\, d\tau \right),\ t\ge a,\ \text{for some}\ a>0,
\end{align}
where $z,y:\mathbb{R}^{+}\to \mathbb{R}^{+}$ are continuous and for $t\to\infty$, $y(t)\to 0$ and $z(t)\to c\in(0,\infty)$.
\end{proposition}
One may take $z$ eventually bounded. We say that
\begin{align*}
\hat{L}(t)=z(t)\exp\left(\int_{a}^{t}\frac{y(\tau)}{\tau}\, d\tau \right),\ t\ge a,
\end{align*}
is {\em normalized slowly varying} at infinity and
\begin{align*}
f(t)=t^{\rho}\hat{L}(t),\ t\ge a,
\end{align*}
is {\em normalized regularly varying} at infinity with index $\rho$ (and denoted by $f\in NRV_{\rho}$). 
A function $f\in NRV_{\rho}$ if and only if
\begin{align}\label{eq:def:nrv}
f\in C^{1}[a,\infty)\ \text{for some}\ a>0\ \text{and}\ \lim\limits_{t\to\infty}\frac{tf'(t)}{f(t)}=\rho.
\end{align}
Similarly,
\begin{definition}\label{def:nrvz}
$h$ is called normalized regularly varying at zero with index $\rho$, written as $h\in NRVZ_{\rho}$ if $t\to h(t^{-1})$ belongs to $NRV_{-\rho}$.
\end{definition}
\begin{proposition}(\cite[p-7,18]{sene})\label{prop:reg:lim}
Let $L$ be slowly varying function at infinity. Then
\begin{itemize}
\item[(i)] Any function $0<f\in C^{1}[a,\infty),\ a>0$ satisfying $\lim\limits_{u\to\infty}uf'(u)/f(u)=\rho$ if and only if $f\in NRV_{\rho},\ -\infty<\rho<\infty$.
\item[\namedlabel{itm:p2}{(ii)}]For any $\gamma>0,\ u^{\gamma}L(u)\to\infty,\ u^{-\gamma}L(u)\to 0$ as $u\to\infty$.
\item[(iii)] For $\rho\in\mathbb{R}$ and $u\to\infty,\ \log L(u)/\log u\to 0$ and $\log(u^{\rho}L(u))/\log u\to\rho$.
\end{itemize}
\end{proposition}
\begin{remark}
If $f\in NRV_{\rho}$ then its primitive $F\in NRV_{\rho+1}$.
\end{remark}
\begin{proposition}\label{prop:assymp:reg:func}(Assymptotic behavior~\cite[Theorem 2.1, p-53]{sene})
If $L$ is a slowly varying function at infinity, then for $a\ge 0$ and $u\to\infty$, we have
\begin{itemize}
\item[(i)] $\int_{a}^{u}s^{\rho}L(s)\, ds \cong (\rho +1)^{-1}u^{\rho+1}L(u)$, for $\rho>-1$;
\item[(ii)] $\int_{t}^{\infty}s^{\rho}L(s)\, ds \cong (-\rho-1)^{-1}u^{\rho+1}L(u)$, for $\rho<-1$.
\end{itemize}
\end{proposition}
\begin{definition}
A positive measurable function $f:[a,\infty)\to\mathbb{R}$, for some $a>0$, is called {\em rapidly varying} at infinity if for each $\rho>1$,
\begin{align*}
\lim_{u\to\infty}\frac{f(u)}{u^{\rho}}=\infty.
\end{align*}
\end{definition}
\begin{proposition}
\begin{itemize}
\item[(i)] If a function $f\in NRV_{\rho}$, then $f'\in RV_{\rho-1}$
\item[(ii)] If a function $f$ is rapidly varying at infinity and $f'(u)$ is nondecreasing on $[a,\infty)$, for some $a>0$, then $f'$ is rapidly varying at infinity too.
\end{itemize}
\end{proposition}
\section{First order blow-up behavior}\label{sec:blow:1st}
In this section, we show our main result exposing the first order behavior for problem involving weighted $p$-Laplacian. We need the following auxiliary lemmas in the subsequent steps of the proof.
\begin{lemma}\label{lem:aux:freg}
Let $q=p/(p-1),\ p>1$. If $f\in RV_{\sigma+1}\ (\sigma>p-2)$ is continuous, then
\begin{align*}
\lim_{z\to\infty}\frac{(F(z))^{1/q}}{f(z)\int_{z}^{\infty}(F(s))^{-1/p}\, ds}=\frac{\sigma+2-p}{p(2+\sigma)},
\end{align*}
where $F$ is given in~\eqref{assum:f2}.
\end{lemma}
\begin{proof}
We have from~\eqref{assum:f2} that
\begin{align*}
F(z)=\int_{0}^{z}f(s)\, ds\stackrel{s=tz}{=}\int_{0}^{1}zf(tz)\, dt.
\end{align*}
Therefore, by the L'H\^{o}pitals rule we obtain
\begin{align}\label{eq:f:lim1}
\lim_{z\to\infty}\frac{F(z)}{zf(z)}\stackrel{\left[\frac{\infty}{\infty}\right]}{=}\lim_{z\to\infty}\frac{f(z)}{f(z)+zf'(z)}=\frac{1}{1+\lim\limits_{z\to\infty}zf'(z)/f(z)}\stackrel{\eqref{eq:def:nrv}}{=}\frac{1}{2+\sigma}.
\end{align}
We also note from~\eqref{eq:inv}
\begin{align*}
\int_{\phi(t)}^{\infty}\dfrac{ds}{(qF(s))^{1/p}}= t>0,
\end{align*}
that $\phi$ is a decreasing function such that $\lim\limits_{t\to 0+}\phi(t)=\infty$. Then the direct computation shows that
\begin{align*}
\phi'(t)=-\left(qF(\phi(t)) \right)^{1/p};\quad |\phi'(t)|^{p-2}\phi''(t)=\frac{q}{p}f(\phi(t));\quad\frac{-\phi'(t)}{\phi''(t)}=\frac{p}{q}\frac{(qF(\phi(t)))^{1/q}}{f(\phi(t))}.
\end{align*}
As a consequence of the limit
\begin{align}
\lim_{s\to\infty}\frac{(F(s))^{1/q}}{f(s)}=0,
\end{align}
we note that
\begin{align*}
\lim_{s\to\infty}\frac{\phi'(t)}{\phi''(t)}=0;
\end{align*}
and, by L'H\^{o}pitals rule, that
\begin{align}\label{eq:f:prop1}
\lim_{s\to\infty}\frac{(F(s))^{1/p}}{s}\stackrel{\left[\frac{\infty}{\infty}\right]}{=}\lim_{s\to\infty}\frac{f(s)}{(F(s))^{1/q}}=\infty.
\end{align}
Since $F$ satisfies~\eqref{assum:f2} and~\eqref{eq:f:prop1}, again by applying L'H\^{o}pitals rule we obtain
\begin{align}\label{eq:f:lim2}
\lim_{z\to\infty}\frac{z(F(z))^{-1/p}}{\int_{z}^{\infty}(F(s))^{-1/p}\, ds}\stackrel{\left[\frac{0}{0}\right]}{=}\lim_{z\to\infty}\left(\frac{1}{p}\cdot\frac{zf(z)}{F(z)}-1 \right)\stackrel{\eqref{eq:f:lim1}}{=}\frac{\sigma+2-p}{p}.
\end{align}
Hence the desired limit follows from~\eqref{eq:f:lim1} and~\eqref{eq:f:lim2}.
\end{proof}
On recalling that $\phi'(t)=-\left(qF(\phi(t)) \right)^{1/p}$, we have the following useful corollary.
\begin{corollary}\label{cor:aux:freg}
Let $q=p/(p-1),\ p>1$. If $f\in RV_{\sigma+1}\ (\sigma>p-2)$ is continuous, then
\begin{align}
\lim_{t\to 0+}\frac{\Phi_{p}(\phi'(t))}{tf(\phi(t))}=-\frac{q}{p}\frac{\sigma+2-p}{2+\sigma}.
\end{align}
\end{corollary}
We would like to note down the following lemma to be useful in Theorem~\ref{thm:main:gen}.
\begin{lemma}\label{lem:aux:def:plap}
For $z,d\in C^{2}(\Omega),\ \phi\in C^{2}(\mathbb{R})$ and $v(x)=\phi(z(x))$
\begin{align*}
\Delta_{d^{\alpha},p} v(x):={\rm div}((d(x))^{\alpha}\Phi_{p}(\nabla \phi(z(x))))={\rm div}((d(x))^{\alpha}\Phi_{p}(\phi'(z(x)))\Phi_{p}(\nabla z(x)))\\
=|\phi'(z(x))|^{p-2}\left[(p-1)(d(x))^{\alpha}\phi''(z(x))|\nabla z(x)|^{p} +\phi'(z(x)) \Delta_{d^{\alpha},p}z(x)\right].
\end{align*}
\end{lemma}

We now proof one of our main theorems.
\begin{proof}[Proof of Theorem~\ref{thm:main:gen}]
{\em Existence:}
For given $\delta>0$, we denote
\begin{align*}
\B_{\delta}:=\{x\in\Omega:\ 0<d(x)<\delta \}.
\end{align*}
Thus for a $C^{2}$-smooth bounded domain $\B$, there exists a constant $\mu>0$, depending on $\B$ such that
\begin{align}\label{eq:dist:pr}
d\in C^2(\overline{\B}_{\mu})\quad \text{and}\quad |\nabla d|\equiv 1\ \text{on}\ \B_{\mu}.
\end{align}
Let us define for $\rho\in(0,\mu/2)$,
\begin{align*}
\B^{-}_{\rho}:=\B_{\mu}\setminus \bar{\B}_{\rho},\quad \B^{+}_{\rho}:=\B_{\mu-\rho},\quad\text{and}\quad d^{\pm}(x):=d(x)\pm \rho
\end{align*}
{\sc Case I.} Let us consider the case when $k\in\mathcal{K}$ (see Section~\ref{sec:intro}), is {\em nondecreasing} on $(0,\nu)$ for some $\nu>0$. We can take, without loss of generality, that $\nu>\mu$.
\begin{align}\label{eq:subs:p}
z^{\pm}(x):=K(d^{\pm}(x)),\ x\in \B^{\pm}_{\rho};\quad
w^{\pm}:=\xi^{\pm}\phi(z^{\pm}(x)),\ x\in \B^{\pm}_{\rho},
\end{align}
where
\begin{align}\label{eq:lim:const}
\xi^{+}:=\left[\dfrac{ p+l_1(1-\alpha)(2+\sigma-p)}{(b_2+ 2\varepsilon)(2+\sigma)} \right]^{1/(2+\sigma -p)},\quad
\xi^{-}:=\left[\dfrac{ p+l_1(1-\alpha)(2+\sigma-p)}{(b_1- 2\varepsilon)(2+\sigma)} \right]^{1/(2+\sigma -p)},
\end{align}
for $\varepsilon\in (0,\min\{1/2,b_1/2\})$.

It follows from~\eqref{eq:dist:pr} and~\ref{assum:B2} that, for such $\varepsilon>0$ considered in~\eqref{eq:lim:const}, there exists $\delta_{\varepsilon}\in(0,\mu/2)$ such that for $\rho\in(0,\delta_{\varepsilon})$,
\begin{align}
(b_1-\varepsilon)d^{\alpha-\frac{\alpha p}{2}}k^{p}(d(x)-\rho)\le (b_1-\varepsilon)d^{\alpha-\frac{\alpha p}{2}}k^{p}(d(x))<b(x),\quad x\in \B^{-}_{\rho};\label{eq:b:lim1}\\
b(x)<(b_2+\varepsilon)d^{\alpha-\frac{\alpha p}{2}}k^{p}(d(x))\le (b_2-\varepsilon)d^{\alpha-\frac{\alpha p}{2}}k^{p}(d(x)+\rho),\quad x\in \B^{+}_{\rho}.\label{eq:b:lim2}
\end{align}
We also have from~\eqref{eq:subs:p} and~\eqref{eq:cls:k} and the successive computation that
\begin{align*}
|\nabla z^{\pm}|^{p}=|K'(d)|^{p}|\nabla d|^{p}=d^{-\alpha p/2}k^{p}(d)|\nabla d|^{p}.
\end{align*}
Then Lemma~\ref{lem:aux:def:plap}, applying for for $v^{\pm}:=\phi(z^{\pm})$ and $z^{\pm}=K(d^{\pm})$ successively, shows that
\begin{align*}
\Delta_{d^{\alpha},p}v^{\pm} 
={\rm div}(d^{\alpha}\Phi_{p}(\phi'(z^{\pm}))\Phi_{p}(\nabla z^{\pm}))~~~~~~~~~~~~~~~~~~~~~~~~~~~~~~~~~~~~~~~~~~~~~~~~~~~~~~~~~~\\
\stackrel{{\rm Lemma~\ref{lem:aux:def:plap}\ for}\ v^{\pm}(x)}{=}(p-1)d^{\alpha}|\phi'(z^{\pm})|^{p-2}\phi''(z^{\pm})|\nabla z^{\pm}|^{p}+\Phi_{p}(\phi'(z^{\pm}))\Delta_{d^{\alpha},p}z^{\pm}\\
\stackrel{{\rm Lemma~\ref{lem:aux:def:plap}\ for}\ z^{\pm}(x)}{=}(p-1)d^{\alpha}|K'(d^{\pm})|^{p}|\phi'(K(d^{\pm}))|^{p-2}\phi''(K(d^{\pm}))|\nabla d|^{p}~~~~~~~~~~~~~~~~~~~~\\
+\Phi_{p}(\phi'(z^{\pm}))\left[(p-1)d^{\alpha}|K'(d^{\pm})|^{p-2}K''(d^{\pm})|\nabla d|^{p}+\Phi_{p}(K'(d^{\pm}))\Delta_{d^{\alpha},p}d^{\pm} \right]\\
=d^{\alpha}|K'(d^{\pm})|^{p}\left[(p-1)|\phi'(K(d^{\pm})))|^{p-2}\phi''(K(d^{\pm}))|\nabla d|^{p} 
+\Phi_{p}(\phi'(K(d^{\pm})))\frac{\Delta_{p}d^{\pm}}{K'(d^{\pm})} \right.\\
\left. +\Phi_{p}(\phi'(K(d^{\pm})))\left((p-1)\frac{K''(d^{\pm})}{(K'(d^{\pm}))^2} + \frac{\alpha}{dK'(d)}\right)|\nabla d|^{p} \right].
\end{align*}
Then, together with $p-1=p/q$, we obtain
 \begin{align}\label{eq:subs:main}
 b(x)f(w^{\pm})-\Delta_{d^{\alpha},p}w^{\pm}=b(x)f(w^{\pm})-(\xi^{\pm})^{p-1}\Delta_{d^{\alpha},p}v^{\pm}~~~~~~~~~~~~~~~~~~~~~~~~~~~~~~~~~~~~~~~~~~~~~~~\nonumber\\
  =(\xi^{\pm})^{p-1}k^{p}(d)f(\phi(z^{\pm}))d^{\alpha-\frac{\alpha p}{2}}\left[\frac{b(x)f(\xi^{\pm}\phi(z^{\pm}))}{(\xi^{\pm})^{p-1}d^{\alpha}|K'(d^{\pm})|^{p}f(\phi(z^{\pm}))}
 -\left( \mathcal{D}_{1}^{\pm}
+\mathcal{D}_{2}^{\pm}
+\mathcal{D}_{3}^{\pm}
  \right)\right],
 \end{align}
 where
 \begin{align*}
 \mathcal{D}_{1}^{\pm}(x)&:= \frac{\Phi_{p}(\phi'(z^{\pm}(x)))}{K'(d^{\pm}(x))f(\phi(z^{\pm}(x)))}\Delta_{p}d(x),\\
 \mathcal{D}_{2}^{\pm}(x)&:=\frac{p}{q}
\left\{ 
\frac{K''(d^{\pm})}{(K'(d^{\pm}(x)))^2}\frac{\Phi_{p}(\phi'(z^{\pm}(x)))}{f(\phi(z^{\pm}(x)))}+\frac{|\phi'(z^{\pm}(x))|^{p-2}\phi''(z^{\pm}(x))}{f(\phi(z^{\pm}(x)))}
 \right\}
 |\nabla d(x)|^{p},\\
 \mathcal{D}_{3}^{\pm}(x)&:=  \alpha\frac{\Phi_{p}(\phi'(z^{\pm}(x)))}{d^{\pm}(x)K'(d^{\pm}(x))f(\phi(z^{\pm}(x)))}|\nabla d(x)|^{p}.
 \end{align*}
 Recalling from~\eqref{eq:subs:p} that $z^{\pm}(x)=K(d^{\pm}(x))$, and we obtain 
\begin{align*}
|\mathcal{D}_{1}^{\pm}(x)|&=\left| \frac{\Phi_{p}(\phi'(z^{\pm}(x)))}{K'(d^{\pm}(x))f(\phi(z^{\pm}(x)))}\Delta_{p}d(x)\right|
= \frac{K(d^{\pm}(x))}{K'(d^{\pm}(x))}\left| \frac{\Phi_{p}(\phi'(z^{\pm}(x)))}{z^{\pm}(x)f(\phi(z^{\pm}(x)))}\right||\Delta_{p}d(x)|.
\end{align*}
Since $d\in C^{2}(\overline{\B}_{\mu})$ and $k\in\mathcal{K}$, we eventually have by Corollary~\ref{cor:aux:freg} and Remark~\ref{rem:k:lim},
\begin{align*}
\lim_{(d(x),\rho)\to (0+,0+)}\mathcal{D}_{1}^{\pm}(x)=0.
\end{align*}
Since  $\nabla d\equiv 1$ on $\B_{\mu}$, and from~\eqref{eq:inv} that $|\phi'(z^{\pm}(x))|^{p-2}\phi''(z^{\pm}(x))=(q/p)f(\phi(z^{\pm}(x)))$, we have
\begin{align*}
\mathcal{D}_{2}^{\pm}(x)
=\frac{p}{q}\frac{K''(d^{\pm})K(d(x))}{(K'(d^{\pm}(x)))^2}\cdot\frac{\Phi_{p}(\phi'(z^{\pm}(x)))}{K(d(x))f(\phi(z^{\pm}(x)))}+1.
\end{align*}
Since we have  $z^{-}\le K(d)$ and $z^{+}\ge K(d)$. Therefore,
\begin{align*}
\mathcal{D}_{2}^{-}(x)\le \frac{p}{q}\frac{K''(d^{\pm}(x))K(d(x))}{(K'(d^{\pm}(x)))^2}\frac{\Phi_{p}(\phi'(z^{-}(x)))}{z^{-}(x)f(\phi(z^{-}(x)))}+1=:\tilde{\mathcal{D}}_{2}^{-}(x);\\
\mathcal{D}_{2}^{+}(x)\ge \frac{p}{q}\frac{K''(d^{\pm}(x))K(d(x))}{(K'(d^{\pm}(x)))^2}\frac{\Phi_{p}(\phi'(z^{+}(x)))}{z^{+}(x)f(\phi(z^{+}(x)))}+1=:\tilde{\mathcal{D}}_{2}^{+}(x).
\end{align*}
Furthermore, we have
\begin{align*}
\mathcal{D}_{3}^{\pm}(x)&
=  \alpha\frac{K(d(x))/K'(d^{\pm}(x))}{d(x)}\frac{\Phi_{p}(\phi'(z^{\pm}(x)))}{K(d(x))f(\phi(z^{\pm}(x)))}\\
&\gtreqless \alpha\frac{K(d(x))/K'(d^{\pm}(x))}{d(x)}\frac{\Phi_{p}(\phi'(z^{\pm}(x)))}{z^{\pm}(x)f(\phi(z^{\pm}(x)))}=:\tilde{\mathcal{D}}_{3}^{\pm}(x).
\end{align*}

Therefore, for $x\in\B_{\rho}^{\pm}$ along with~\eqref{eq:subs:main},~\eqref{eq:b:lim1} and~\eqref{eq:b:lim2}, we conclude that
\begin{align}
 b(x)f(w^{-})-\Delta_{d^{\alpha},p}w^{-}~~~~~~~~~~~~~~~~~~~~~~~~~~~~~~~~~~~~~~~~~~~~~~~~~~~~~~~~~~~~~~~~~~~~~~~~~~~~~~~~~~~~~~~\nonumber\\
 \ge (\xi^{-})^{p-1}k^{p}(d)f(\phi(z^{-}))d^{\alpha-\frac{\alpha p}{2}}\left\{\frac{(b_1-\varepsilon)f(\xi^{-}\phi(z^{-}))}{(\xi^{-})^{p-1}f(\phi(z^{-}))}
 -\left( \mathcal{D}_{1}^{-}(x)
+\tilde{\mathcal{D}}_{2}^{-}(x)
+\mathcal{D}_{3}^{-}(x)
  \right)\right\};\label{eq:super:constr}\\
   b(x)f(w^{+})-\Delta_{d^{\alpha},p}w^{+}~~~~~~~~~~~~~~~~~~~~~~~~~~~~~~~~~~~~~~~~~~~~~~~~~~~~~~~~~~~~~~~~~~~~~~~~~~~~~~~~~~~~~~~\nonumber\\
 \le (\xi^{+})^{p-1}k^{p}(d)f(\phi(z^{+}))d^{\alpha-\frac{\alpha p}{2}}\left\{\frac{(b_2+\varepsilon)f(\xi^{+}\phi(z^{+}))}{(\xi^{+})^{p-1}f(\phi(z^{+}))}
 -\left( \mathcal{D}_{1}^{+}(x)
+\tilde{\mathcal{D}}_{2}^{+}(x)
+\mathcal{D}_{3}^{+}(x)
  \right)\right\}\label{eq:sub:constr}
\end{align}
Note that by Corollary~\ref{cor:aux:freg} and Remark~\ref{rem:k:lim}, we have
\begin{align*}
\lim_{(d(x),\rho)\to (0+,0+)} \tilde{\mathcal{D}}_{2}^{\pm}(x)=\frac{p+l_1(2+\sigma - p)}{2+\sigma};\\
\lim_{(d(x),\rho)\to (0+,0+)} \tilde{\mathcal{D}}_{3}^{\pm}(x)=-\alpha\frac{l_1(2+\sigma - p)}{2+\sigma}.
\end{align*}
Thus as $(d(x),\rho)\to (0+,0+)$, the expression in the first bracket $(\cdot)$ in~\eqref{eq:super:constr}-\eqref{eq:sub:constr} converges to
\begin{align*}
\frac{p+l_1(1-\alpha)(2+\sigma - p)}{2+\sigma}
\end{align*}
Therefore, as $(d(x),\rho)\to (0+,0+)$, the expression in $\{\cdot\}$ in~\eqref{eq:super:constr}-\eqref{eq:sub:constr} converges, respectively, to
\begin{align*}
\left(\frac{b_1-\varepsilon}{b_1 - 2\varepsilon} -1\right)\frac{p+l_1(1-\alpha)(2+\sigma - p)}{2+\sigma}> 0,\\
\text{and}\qquad\left(\frac{b_2+\varepsilon}{b_2 + 2\varepsilon} -1\right)\frac{p+l_1(1-\alpha)(2+\sigma - p)}{2+\sigma}< 0.
\end{align*}
Therefore, using the respective continuity (limits) of the absorption coefficient, we obtain
\begin{align}\label{eq:sub:sup:constr}
b(x)f(w^{\pm})-\Delta_{d^{\alpha},p}w^{\pm}\lessgtr 0,\ \text{respectively on}\ \B^{\pm}_{\rho},
\end{align}
where $\rho\in(0,\mu/2)$  for sufficiently small $\mu>0$.


Now suppose that $u$ is a nonnegative solution of~\eqref{eq:1}-\eqref{eq:2}. First we note that for $M$ sufficiently large
\begin{align*}
u\le w^{-}+M,\quad\text{on}\ \partial\B_{2\delta_{\varepsilon}},\ \text{and}\quad 
w^{+}\le u+M,\quad\text{on}\ \partial\B_{2\delta_{\varepsilon}-\rho}.
\end{align*}
We observe that $w^{-}(x)\to\infty$ as $d^{-}(x)\to \rho$, and $u_{|_{\partial\B}}=+\infty>w^{+}(x)_{|_{\partial\B}}$. It follows from comparison principle in Theorem~\ref{thm:comp:1}, that
\begin{align*}
u\le w^{-}+M,\quad\text{on}\ \B_{\rho}^{-},\quad \text{and}\ 
w^{+}\le u+M,\quad\text{on}\ \B_{\rho}^{+}.
\end{align*}
Hence, for $x\in\B^{-}_{\rho}\cap\B^{+}_{\rho}$, we have
\begin{align*}
\xi^{+}-\frac{M}{\phi(K(d^{\pm}(x)))}\le \frac{u(x)}{\phi(K(d^{\pm}(x)))}\le \xi^{-}+\frac{M}{\phi(K(d^{\pm}(x)))}
\end{align*}
On letting $\rho\to 0$, we see that
\begin{align*}
\xi^{+}-\frac{M}{\phi(K(d(x)))}\le \frac{u(x)}{\phi(K(d(x)))}\le \xi^{-}+\frac{M}{\phi(K(d(x)))}.
\end{align*}
On recalling that $\phi(t)\to \infty$ as $t\to 0$, we have
\begin{align}\label{eq:growth2}
\xi^{+}\le \liminf_{d(x)\to 0}\frac{u(x)}{\phi(K(d(x)))}
\le \limsup_{d(x)\to 0}\frac{u(x)}{\phi(K(d(x)))}
\le \xi^{-}.
\end{align}
The claimed result in~\eqref{eq:growth1} follows as $\varepsilon\to 0$ in~\eqref{eq:growth2}, whence for $b_1=b_2=c$,~\eqref{eq:growth} follows as well.

\noindent
{\sc Case II.} Let us consider the case when $k\in\mathcal{K}$ (see Section~\ref{sec:intro}), is {\em nonincreasing} on $(0,\nu)$ for some $\nu>0$.
Let
\begin{align*}
y^{\pm}(x):=K(d(x))\pm K(\rho),\quad x\in \B^{\pm}_{\rho}.
\end{align*}
Given $0<\varepsilon<c/2$, it follows from~\eqref{eq:dist:pr} and~\eqref{assum:b1} that there exists $\delta_{\varepsilon}\in(0,\mu/2)$,
\begin{align}
(c-\varepsilon)d^{\alpha-\frac{\alpha p}{2}}k^{p}(d(x))<b(x)<(c+\varepsilon)d^{\alpha-\frac{\alpha p}{2}}k^{p}(d(x)),\quad x\in \B^{+}_{\rho}.\label{eq:b:lim3}
\end{align}
In a similar way, as of {\sc Case I}, we can show~\eqref{eq:sub:sup:constr} hold for $w^{\pm}:=\xi^{\pm}\phi(y^{\pm}(x))$.
\end{proof}
We proof now uniqueness Theorem~\ref{thm:unique} under the additional assumption~\eqref{eq:cond:f:unique}.
\begin{proof}[Proof of Theorem~\ref{thm:unique}]
The uniqueness follows from Theorem~\ref{thm:main:gen} by a standard argument using the blow-up rate on the boundary. Indeed, suppose $u_1$ and $u_2$ are two solutions of~\eqref{eq:1} in $\B$. Then by Theorem~\ref{thm:main:gen}, it follows that
\begin{align*}
\lim_{d(x)\to 0}\frac{u_1(x)}{u_2(x)}=1.
\end{align*}
Thus by the definition of limit, for every $\varepsilon>0$, we can find a $\delta>0$ (as small as we please) such that
\begin{align*}
(1-\varepsilon)u_2(x)\le u_1(x)\le (1+\varepsilon)u_2(x),\quad\in\B_{\delta}.
\end{align*}
Since $f$ satisfies~\eqref{eq:cond:f:unique}, we can see that $U_{\pm}:=(1\pm\varepsilon)u_2(x),\ x\in\B$, satisfy
\begin{align*}
{\rm div}(d^{\alpha}\Phi_{p}(\nabla U_{+})) \le b(x)f(U_{+})\quad\text{and}\quad {\rm div}(d^{\alpha}\Phi_{p}(\nabla U_{-})) \ge b(x)f(U_{-})\qquad \text{in}\ \B.
\end{align*}
Now, consider the following problem
\begin{align}\label{eq:14}
{\rm div}(d^{\alpha}\Phi_{p}(\nabla w)) = b(x)f(w),\ \text{in}\ \B_{0},\quad w=u\ \text{on}\ \partial\B_{0},
\end{align}
where $\B_{0}:=\B\setminus\B_{\delta}$. Thanks to Theorem~\ref{thm:1eaux},~\eqref{eq:14} possesses a unique solution, necessarily $u$.  Then by comparison principle by Theorem~\ref{thm:comp:1}, it follows that
\begin{align*}
U_{-}(x)\le u(x)\le U_{+}(x),\quad x\in\B_{0}.
\end{align*}
It is evident that $u=u_1$ on $\B_{0}$, thereby we have
\begin{align*}
(1-\varepsilon)u_2(x)\le u_1(x)\le (1+\varepsilon)u_2(x),\quad\in\B_{\delta}\cup\B_{0}.
\end{align*}
Letting $\varepsilon\to 0$, we arrive at $u_1 =u_2$.
\end{proof}

\section{Second order blow-up behavior}\label{sec:2nd:blw}
Here we give a second order estimates for large solutions of semilinear problem, for $p=2$ in Theorem~\ref{thm:p:2} and~\ref{thm:p:2a} respectively for the growths like $(-\ln (d(x)))^{\tau},\ \tau>0$ and $(d(x))^{\omega}, \omega>0$. 
We shall collect some auxiliary results to be useful in the subseuqent Lemmas and Theorems.
\begin{lemma}\label{lem:k:prop}
Let $k\in\mathcal{K}_{l_1}$. Then
\begin{enumerate}
\item[\namedlabel{itm:k:prop1}{(i)}] $\lim\limits_{t\to 0+}\frac{K(t)}{t^{-\alpha/2}k(t)}=0$, $\lim\limits_{t\to 0+}\frac{K(t)}{t^{1-\alpha/2}k(t)}=l_1$, i.e., $K\in NRVZ_{l_1^{-1}}$.
\item[\namedlabel{itm:k:prop2}{(ii)}] $\lim\limits_{t\to 0+}\frac{t(t^{-\alpha/2}k(t))'}{t^{-\alpha/2}k(t)}=l_1^{-1}-1$, i.e., $k\in NRVZ_{l_1^{-1}-1+\alpha/2}$;
\item[\namedlabel{itm:k:prop3}{(iii)}] $\lim\limits_{t\to 0+}\frac{1}{y(t)}\left(\frac{K(t)(t^{-\alpha/2}k(t))'}{t^{-\alpha}k^2(t)}-(1-l_1)\right)=-e_k$
\end{enumerate}
\end{lemma}
\begin{proof}
By definition of the class $\mathcal{K}$, Remark~\ref{rem:k:lim} and applying of L'H\"o{}spital rule, we obtain\\
\ref{itm:k:prop1}: $\lim\limits_{t\to 0+}\frac{K(t)}{t^{-\alpha/2}k(t)}=\lim\limits_{t\to 0+}\frac{K(t)}{K'(t)}=0$ and  $\lim\limits_{t\to 0+}\frac{K(t)}{t^{1-\alpha/2}k(t)}=\lim\limits_{t\to 0+}\dfrac{\frac{K(t)}{K'(t)}}{t}\stackrel{\left[\frac{0}{0}\right]}{=}\lim\limits_{t\to 0+}\left(\frac{K(t)}{t^{-\alpha/2}k(t)}\right)^{(1)}=l_1$.\\
\ref{itm:k:prop2}-\ref{itm:k:prop3}: We note that $\left(\frac{K(t)}{t^{1-\alpha/2}k(t)}\right)^{(1)}=1-\frac{K(t)K''(t)}{(K'(t))^2}$ and $\frac{t(t^{-\alpha/2}k(t))'}{t^{-\alpha/2}k(t)}=\frac{tK''(t)}{K'(t)}$. Then\\
$\lim\limits_{t\to 0+}\frac{t(t^{-\alpha/2}k(t))'}{t^{-\alpha/2}k(t)}=\lim\limits_{t\to 0+}\frac{K(t)K''(t)}{(K'(t))^2}\lim\limits_{t\to 0+}\frac{tK'(t)}{K(t)}\stackrel{\eqref{eq:cls:k},\ref{itm:k:prop1}}{=}\frac{1-l_1}{l_1}$. This implies that $K'\in NRVZ_{l_1^{-1}-1}$, and hence $k\in NRVZ_{l_1^{-1}-1+\alpha/2}$. Furthermore,
\begin{align*}
\lim\limits_{t\to 0+} \frac{1}{y(t)}\left(\frac{K(t)(t^{-\alpha/2}k(t))'}{t^{-\alpha}k^2(t)}-(1-l_1)\right)=-\lim\limits_{t\to 0+} \frac{\left(\frac{K(t)}{K'(t)}\right)'-l_1}{y(t)}=-e_k.
\end{align*}
\end{proof}
Note that
\begin{align}\label{eq:phi}
\int_{\phi(t)}^{\infty}\frac{d\tau}{\sqrt{2F(\tau)}}=t,\ \forall\ t>0,
\end{align}
then
\begin{align}\label{eq:phi:1a}
\int_{t}^{\infty}\frac{d\tau}{\sqrt{2F(\tau)}}=\phi^{-1}(t),\ \forall\ t>0;\quad\text{and}\quad (\phi^{-1}(t))'=-\frac{1}{\sqrt{2F(t)}},\ \forall\ t>0.
\end{align}
\begin{description}
\item[\namedlabel{itm:y}{(Y)}]
$y(t)\in C([0,\nu))$ is a nondecreasing function such that $y(0)=0$ and $t/y(t)\rightarrow 0$ as $t\to 0+$
\end{description}
\begin{lemma}
If $f\in RV_{\sigma +1},\ \sigma>p-2$ and it satisfies the assumptions as in~\ref{itm:f1},
 then
\begin{enumerate}
\item[\namedlabel{itm:f:prop1}{(i)}] $\lim\limits_{t\to\infty}\frac{tf(t)}{F(t)}=\sigma +2$;
\item[\namedlabel{itm:f:prop2}{(ii)}] $\phi^{-1}\in NRV_{-\sigma/2},\ i.e.,\ \lim\limits_{t\to\infty}\frac{t(\phi^{-1}(t))'}{\phi^{-1}(t)}=\lim\limits_{t\to\infty}\frac{t}{\phi^{-1}(t)\sqrt{2F(t)}}=-\frac{\sigma}{2}$;
\item[\namedlabel{itm:f:prop3}{(iii)}] $\lim\limits_{t\to\infty}\frac{\sqrt{2F(t)}}{f(t)\phi^{-1}(t)}=\frac{\sigma}{\sigma +2}$.
\end{enumerate}
\end{lemma}
\begin{proof}
By Proposition~\ref{prop:assymp:reg:func}, we have for $t\to\infty$,
\begin{align*}
F(t)=\frac{t^{\sigma+2}}{\sigma+2}\hat{L}(t);\quad (2F(t))^{-1/2}=\left(\frac{2t^{\sigma+2}}{\sigma+2}\hat{L}(t)\right)^{-1/2};\quad\phi^{-1}(t)=\left(\frac{\sigma^{2}t^{\sigma}}{2(\sigma+2)}\hat{L}(t)\right)^{-1/2}.
\end{align*}
\end{proof}

\begin{lemma}\label{lem:f1}
If $f\in RV_{\sigma +1},\ \sigma>p-2$ and it satisfies the assumptions as in~\ref{itm:f1},
 then
\begin{enumerate}
\item[\namedlabel{itm:f1:prop1}{(i)}] $\lim\limits_{t\to\infty}\frac{\frac{tf'(t)}{f(t)}-(\sigma+1)}{\phi^{-1}(t)}=0$;
\item[\namedlabel{itm:f1:prop2}{(ii)}] $\lim\limits_{t\to\infty}\frac{\frac{F(t)}{tf(t)}-\frac{1}{\sigma+2}}{\phi^{-1}(t)}=0$;
\item[\namedlabel{itm:f1:prop3}{(iii)}] $\lim\limits_{t\to\infty}\frac{\frac{\sqrt{2F(t)}}{f(t)\phi^{-1}(t)}-\frac{\sigma +1}{\sigma+2}}{\phi^{-1}(t)}=0$;
\item[\namedlabel{itm:f1:prop4}{(iv)}] $\lim\limits_{t\to\infty}\frac{\frac{f(\xi_{0}t)}{\xi_{0}f(t)}-\xi_{0}^{\sigma+1}}{\phi^{-1}(t)}=0,\  \xi_{0}>0$.
\end{enumerate}
\end{lemma}
\begin{proof}
By Proposition~\ref{prop:assymp:reg:func}, we have for $t\to\infty$,
\begin{align*}
F(t)=\frac{t^{\sigma+2}}{\sigma+2}\hat{L}(t);\quad (2F(t))^{-1/2}=\left(\frac{2t^{\sigma+2}}{\sigma+2}\hat{L}(t)\right)^{-1/2};\quad\phi^{-1}(t)=\left(\frac{\sigma^{2}t^{\sigma}}{2(\sigma+2)}\hat{L}(t)\right)^{-1/2}.
\end{align*}
\end{proof}

\begin{lemma}\label{lem:phi:prop}
Let us assume the hypothesis of Theorem~\ref{thm:p:2a}, assumption~\ref{itm:y} and $\phi$ satisfies~\eqref{eq:phi}. Then
\begin{enumerate}
\item[\namedlabel{itm:phi:prop1}{(i)}] $-\phi'(t)=\sqrt{2F(\phi(t))},\ \phi(t)>0,\ t>0,\ \phi(0):=\lim\limits_{t\to 0+}\phi(t)=+\infty,\ \phi''(t)=f(\phi(t)),\ t>0$;
\item[\namedlabel{itm:phi:prop2}{(ii)}] $\lim\limits_{t\to 0+}\frac{t\phi'(t)}{\phi(t)}=-\frac{2}\sigma,\ i.e.,\ \phi\in NRVZ_{-2/\sigma}$;
\item[\namedlabel{itm:phi:prop3}{(iii)}] $\lim\limits_{t\to 0+}\frac{\phi'(t)}{t\phi''(t)}=-\frac{\sigma}{\sigma+2},\ i.e.,\ \phi\in NRVZ_{-(\sigma+2)/\sigma}$;
\item[\namedlabel{itm:phi:prop4}{(iv)}] $\lim\limits_{t\to 0+}\frac{\phi(t)}{t^{2}\phi''(t)}=\frac{\sigma^{2}}{2(\sigma+2)}$;
\item[\namedlabel{itm:phi:prop5}{(v)}] $\lim\limits_{t\to 0+}\frac{\frac{\phi'(t)}{t\phi''(t)}+\frac{\sigma+2}{\sigma}}{t}=0$;
\item[\namedlabel{itm:phi:prop6}{(vi)}] for $k\in\mathcal{K},\ \lim\limits_{t\to 0+}\frac{1}{y(t)}\left[1+\frac{\phi'(K(t))}{K(t)\phi''(K(t))}\left(\frac{K(t)K''(t)}{(K'(t))^{2}}++\alpha l_1\right)-\frac{f(\xi_{0}\phi(K(t)))}{\xi_{0}f(\phi(K(t)))}\right]=\frac{\sigma}{\sigma+2}e_k$.
\end{enumerate}
\end{lemma}
\begin{proof}[Proof of Lemma~\ref{lem:phi:prop}]
\ref{itm:phi:prop1} It follows from the definition in~\eqref{eq:phi}.
For \ref{itm:phi:prop2}-\ref{itm:phi:prop5}, let $u=\phi(t)$ and the successive use of L'H\"o{}spital rule and Lemma~\ref{lem:f1} give us
\begin{align*}
&\ref{itm:phi:prop2} \lim\limits_{t\to 0+}\frac{t\phi'(t)}{\phi(t)}=
-\lim\limits_{t\to 0+}\frac{t\sqrt{2F(\phi(t))}}{\phi(t)}=-\lim\limits_{u\to \infty}\frac{\sqrt{2F(u)}\int_{u}^{\infty}\frac{ds}{\sqrt{2F(s)}}}{u}=-\frac{2}{\sigma};\\
&\ref{itm:phi:prop3} \lim\limits_{t\to 0+}\frac{\phi'(t)}{t\phi''(t)}=-\lim\limits_{t\to 0+}\frac{\sqrt{2F(\phi(t))}}{tf(\phi(t))}=-\lim\limits_{u\to \infty}\frac{\sqrt{2F(u)}}{f(u)\phi^{-1}(u)}=-\frac{\sigma}{\sigma+2};\\
&\ref{itm:phi:prop4} \lim\limits_{t\to 0+}\frac{\phi(t)}{t^{2}\phi''(t)}=
\lim\limits_{t\to 0+}\frac{\phi(t)}{t\phi'(t)}\cdot \frac{\phi'(t)}{t\phi''(t)}\stackrel{\rm Lemma~\ref{lem:phi:prop}\ref{itm:phi:prop2}-\ref{itm:phi:prop3}}{=}
\frac{\sigma^{2}}{2(\sigma+2)};\\
&\ref{itm:phi:prop5} \lim\limits_{t\to 0+}\frac{\frac{\phi'(t)}{t\phi''(t)}+\frac{\sigma}{\sigma+2}}{t}=-\lim\limits_{u\to \infty}\frac{\frac{\sqrt{2F(u)}}{f(u)\phi^{-1}(u)}-\frac{\sigma}{\sigma+2}}{\phi^{-1}(u)}\stackrel{\rm Lemma~\ref{lem:f1}\ref{itm:f1:prop3}}{=}0.
\end{align*}
For~\ref{itm:phi:prop6}, we note that $K\in NRVZ_{l_1^{-1}}$ and $l_1\in(0,1)$, so $\lim\limits_{t\to 0+}\frac{K(t)}{t}=0$. It follows by the choice of $\xi_{0}$ in~\eqref{eq:xi0}, Lemma~\ref{lem:k:prop}\ref{itm:k:prop3} that
\begin{align*}
& 1-(1-l_1)\frac{\sigma}{2+\sigma}-\alpha l_1\frac{\sigma}{2+\sigma}=\xi_0^{\sigma};\\
& \lim\limits_{t\to 0+}\frac{1}{y(t)}\left(\frac{K(t)(t^{-\alpha/2}k(t))'}{t^{-\alpha}k^2(t)}-(1-l_1)\right)=-e_k ;\\
& \lim\limits_{t\to 0+}\frac{1}{y(t)}\left(\frac{\phi'(K(t))}{K(t)\phi''(K(t))} + \frac{\sigma}{\sigma+2}\right)\\
&~~~~~~~~~~~~~~~~~~~~=\lim_{t\to 0+}\frac{K(t)}{t}\lim_{t\to 0+}\frac{t}{y(t)}\lim\limits_{t\to 0+}\frac{\frac{\phi'(K(t))}{K(t)\phi''(K(t))} + \frac{\sigma}{\sigma+2}}{K(t)}\stackrel{\rm Lemma~\ref{lem:phi:prop}\ref{itm:phi:prop5}}{=}0;\\
& \lim\limits_{t\to 0+}\frac{1}{y(t)}\left(\xi_0^{\sigma}-\frac{f(\xi_0\phi(K(t)))}{\xi_0 f(\phi(K(t)))} \right)\\
&~~~~~~~~~~~~~~~~~~=\lim\limits_{t\to 0+}\frac{K(t)}{t}\lim_{t\to 0+}\frac{t}{y(t)}\lim\limits_{t\to 0+}\frac{\xi_0^{\sigma}-\frac{f(\xi_0\phi(K(t)))}{\xi_0 f(\phi(K(t)))}}{K(t)}\stackrel{\rm Lemma~\ref{lem:f1}\ref{itm:f1:prop4}}{=}0;\\
& \lim\limits_{t\to 0+}\frac{1}{y(t)}\left[1+ \frac{\phi'(K(t))}{K(t)\phi''(K(t))}\left(\frac{K(t)K''(t)}{(K'(t))^2}+\alpha l_1\right)- \frac{f(\xi_0\phi(K(t)))}{\xi_0 f(\phi(K(t)))}\right]\\
& = \lim\limits_{t\to 0+}\left(\frac{\phi'(K(t))}{K(t)\phi''(K(t))}+\frac{\sigma}{\sigma+2} \right)\lim\limits_{t\to 0+}\frac{1}{y(t)}\left(\frac{K(t)(t^{-\alpha/2}k(t))'}{t^{-\alpha}k^2(t)}-(1-l_1)\right)\\
& +(1-(1-\alpha)l_1)\lim\limits_{t\to 0+}\frac{1}{y(t)}\left(\frac{\phi'(K(t))}{K(t)\phi''(K(t))}+\frac{\sigma}{\sigma+2} \right)\\
& - \frac{\sigma}{\sigma+2}\lim\limits_{t\to 0+}\frac{1}{y(t)}\left(\frac{K(t)(t^{-\alpha/2}k(t))'}{t^{-\alpha}k^2(t)}-(1-l_1)\right)+ \lim\limits_{t\to 0+}\frac{1}{y(t)}\left(\xi_0^{\sigma}-\frac{f(\xi_0\phi(K(t)))}{\xi_0 f(\phi(K(t)))} \right)\\
& = \frac{\sigma}{\sigma+2}e_k .
\end{align*}
\end{proof}
\begin{proposition}\label{prop:v}
Suppose $v\in C^{2+\alpha}\cap C^{1}(\overline{\B})$ be the unique solution to the problem
\begin{align}\label{eq:aux}
-\Delta v=1\quad v(x)>0,\ x\in\B,\quad v_{|_{\partial\B}}=0.
\end{align}
By H\"opf maximum principle~\cite[Theorem 3.7]{trudinger2001}, we can have
\begin{align}\label{eq:aux:sol}
\nabla v(x)\neq 0,\ \forall x\in\partial\B\ \text{and}\ c_1 d(x)\le v(x)\le c_2 d(x),\ \forall x\in \B,
\end{align}
for some constants $c_1 , c_2>0$.
\end{proposition}

Now we give the proof of Theorem~\ref{thm:p:2a}.

\noindent
{\em Proof of Theorem~\ref{thm:p:2a}:}
For given $\delta>0$, we denote
\begin{align*}
\B_{\delta}:=\{x\in\B:\ 0<d(x)<\delta \}.
\end{align*}
Thus for a $C^{2}$-smooth bounded domain $\B$, there exists a constant $\mu>0$, depending on $\B$ such that
\begin{align}\label{eq:dist:pr:2}
d\in C^2(\overline{\B}_{\mu})\quad \text{and}\quad |\nabla d|\equiv 1\ \text{on}\ \B_{\mu}.
\end{align}
Let us define for $\rho\in(0,\mu/2)$,
\begin{align*}
\B^{-}_{\rho}:=\B_{\mu}\setminus \bar{\B}_{\rho},\quad \B^{+}_{\rho}:=\B_{\mu-\rho},\quad\text{and}\quad d^{\pm}(x):=d(x)\pm \rho
\end{align*}
Let us consider the case when $k\in\mathcal{K}$ (see Section~\ref{sec:intro}), is {\em nondecreasing} on $(0,\nu)$ for some $\nu>0$. 
We can take, without loss of generality, that $\nu>\mu$.
\begin{align}\label{eq:subs}
z^{\pm}(x):=K(d^{\pm}(x)),\ x\in \B^{\pm}_{\rho};\quad
w^{\pm}:=\xi_{0}\phi(z^{\pm}(x))(1+\chi^{\pm}y(d^{\pm}(x))),\ x\in \B^{\pm}_{\rho},
\end{align}
where
\begin{align}\label{eq:lim:const:1}
\chi^{\pm}=\dfrac{\sigma((1-\alpha/2)e_k-l_1)-B_0(2+\sigma l_1)G(\theta,y)\mp(2+\sigma)\varepsilon}{\sigma(3+\sigma+l_1)-(\alpha/2)\sigma^2 l_1^2+\alpha\sigma l_1}
\end{align}
By the Lagrange mean value theorem applying for $f$ in the interval ending with the endpoints $1$ and $1+\chi^{\pm}y(d^{\pm}(x))$, we obtain,
\begin{align*}
f(w^{\pm}(x))=f(\xi_0 \phi(K(d^{\pm}(x))))+\xi_0 \chi^{\pm}y(d^{\pm}(x))\phi(K(d^{\pm}(x)))f'(\gamma^{\pm}(d^{\pm}(x))),
\end{align*}
where $\gamma^{\pm}(t):=\xi_{0}\phi(K(t))[1+\lambda^{\pm}(t)\chi^{\pm}y(t)]$ and $\lambda^{\pm}(t)\in[0,1]$. Since $f\in RV_{\sigma}$, by Proposition~\ref{prop:unif:conv} we obtain
\begin{align*}
\lim_{d^{\pm}(x)\to 0}\frac{f(\xi_0 \phi(K(d^{\pm}(x))))}{f(\gamma^{\pm}(d^{\pm}(x)))}=\lim_{d^{\pm}(x)\to 0}\frac{f'(\xi_0 \phi(K(d^{\pm}(x))))}{f'(\gamma^{\pm}(d^{\pm}(x)))}=1.
\end{align*}
Then direct computations give the followings.
\begin{align}
\nabla w^{\pm}(x)
=\xi_{0}(A(d^{\pm}(x))+B(d^{\pm}(x)))\nabla d,
\end{align}
where
\begin{align*}
A(t):=\phi'(K(t))K'(t),\qquad B(t):=C^{\pm}\left(\phi'(K(t))K'(t)y(t)+\phi(K(t))y'(t)\right).
\end{align*}
\begin{align}
\Delta w^{\pm}(x)=
\xi_{0}\left[(A'(d(x))+B'(d(x)))|\nabla d|^{2}
+(A(d(x))+B(d(x)))\Delta d\right];\\
\Delta_{d^{\alpha},2}w^{\pm}(x)={\rm div}(d^{\alpha}\nabla w^{\pm}(x))=\alpha d^{\alpha-1}\nabla d\cdot\nabla w^{\pm}(x)
		+d^{\alpha}\Delta w^{\pm}(x)~~~~~~~~~~~~~~~~~~~~~~~~~\nonumber\\
=\xi_{0}\left[\alpha d^{\alpha-1}(A(d(x))+B(d(x)))|\nabla d|^{2}
+d^{\alpha}\left((A'(d(x))+B'(d(x)))|\nabla d|^{2}\right.\right.\nonumber\\
\left.\left. +(A(d(x))+B(d(x)))\Delta d\right)\right].
\end{align}
Then we obtain successively,
\begin{align*}
b(x)f(w^{\pm})-\Delta_{d^{\alpha},2}w^{\pm}=k^{2}(d^{\pm}(x))(1+(c\pm\varepsilon)d^{\pm}(x))f(w^{\pm}(x)) -\Delta_{d^{\alpha},2}w^{\pm}(x)\\
 =\xi_{0}k^{2}(d)\phi''(K(d^{\pm}(x)))y(d^{\pm}(x)))\left[I_1(d^{\pm}(x)) + I_2(d^{\pm}(x)) +I_3(d^{\pm}(x)) +I_4(d^{\pm}(x))\right],
\end{align*}
where
\begin{align*}
I_1(r) &:=\frac{1}{y(r)}\left[1+\frac{\phi'(K(r))}{K(r)\phi''(K(r))}\left(\frac{K(r)K''(r)}{(K'(r))^2}+\alpha l_1\right) -\frac{f(\xi_0\phi(K(r)))}{\xi_0 f(\phi(K(r)))} \right];\\
I_{2}^{\pm}(r) &:=-(B_{0}\pm\varepsilon)\frac{r^{\theta}}{y(r)}\frac{f(\xi_0\phi(K(r)))}{\xi_0 f(\phi(K(r)))}+ \chi^{\pm}\left[1+\frac{\phi'(K(r))}{K(r)\phi''(K(r))}\left(\frac{K(r)K''(r)}{(K'(r))^2}+\frac{2K(r)}{K'(r)}\frac{y'(r)}{y(r)} \right)\right.\\
&\left. ~~~~~+\frac{\phi(K(r))}{(K(r))^{2}\phi''(K(r))}\frac{K(r)K''(r)}{(K'(r))^2}\frac{y''(r)}{y(r)} -\frac{f'(\gamma^{\pm}(K(r)))}{f'(\phi(K(r)))}\frac{\phi(K(r))f(\phi'(K(r)))}{f(\phi(K(r)))}\right];\\
I_3^{\pm}(r) &:=\alpha\left[\frac{\phi'(K(r))}{K(r)\phi''(K(r))}\left(\frac{K(r)}{rK'(r)}\frac{1}{y(r)}-\frac{l_1}{y(r)}+\chi^{\pm}\frac{K(r)}{rK'(r)} \right)\right.\\
&\left. \qquad\qquad\qquad~~~~~~~~~~~~~~~~~~~~~~~~~~~+ \chi^{\pm}\frac{\phi(K(r))}{(K(r))^{2}\phi''(K(r))}\left(\frac{K(r)}{rK'(r)}\right)^{2}\frac{ry'(r)}{y(r)} \right];\\
I_4^{\pm}(r) &:=\frac{\phi'(K(r))}{K(r)\phi''(K(r))}\frac{K(r)}{K'(r)}\left(\frac{1}{y(r)} +\chi^{\pm}\right)+ \chi^{\pm}\frac{\phi(K(r))}{(K(r))^{2}\phi''(K(r))}\left(\frac{K(r)}{K'(r)}\right)^{2}\frac{y'(r)}{y(r)}\\
&~~~~~~~~~~~~~~~~~~~~~~~~~~~~~~~~~~~-\chi^{\pm}(c\pm\varepsilon)\frac{f'(\gamma^{\pm}(K(r)))}{f'(\phi(K(r)))}\frac{\phi(K(r))f(\phi'(K(r)))}{f(\phi(K(r)))}.
\end{align*}
Thus together with Lemma~\ref{lem:phi:prop} and suitable choices of $\xi_0$ and $\chi$, we have
\begin{lemma}\label{lem:lims:ln}
Let the assumptions of Theorem~\ref{thm:p:2} and $y(r)=(-\ln r)^{-\tau},\ \tau>0$. Then
\begin{align*}
&(i)\ \lim\limits_{r\to 0+}I_1(r)=\frac{\sigma}{2+\sigma}e_k;\quad (ii)\ \lim\limits_{r\to 0+}I_{2}^{\pm}(r)=\chi^{\pm}\left[1-\frac{\sigma}{2+\sigma}(1-l_1)-\sigma\right];\\
&(iii)\ \lim\limits_{r\to 0+}I_3^{\pm}(r)=-\frac{\alpha\sigma}{2+\sigma}(L^{*}+\chi^{\pm}l_1);\ (iv)\ \lim\limits_{r\to 0+}I_4^{\pm}(r)=0.
\end{align*}
\end{lemma}
\begin{lemma}\label{lem:lims}
Let the assumptions from Theorem~\ref{thm:p:2a}. Then
\begin{align*}
&(i)\ \lim\limits_{r\to 0+}I_1(r)=\frac{\sigma}{2+\sigma}e_k;\\
&(ii)\ \lim\limits_{r\to 0+}I_{2}^{\pm}(r)=\chi^{\pm}\left[1-\frac{\sigma}{2+\sigma}(1+l_1)-\sigma\right]-(B_{0}\pm\varepsilon)\frac{2+l_1(1-\alpha)\sigma}{2+\sigma}G(\theta,y);\\
&(iii)\ \lim\limits_{r\to 0+}I_3^{\pm}(r)=\alpha\left[-\frac{\sigma}{2+\sigma}\frac{e_k}{2}+\chi^{\pm}\left( \frac{\sigma^2 l_1^2}{2(2+\sigma)}-\frac{\sigma l_1}{2+\sigma}\right)\right];\ (iv)\ \lim\limits_{r\to 0+}I_4^{\pm}(r)=-\frac{\sigma l_1}{2+\sigma}.
\end{align*}
\end{lemma}
Fix $\varepsilon>0$ and choose $\delta>0$ such that
\begin{enumerate}
\item $d(x)\in C^{2}(\{x\in\B:\ d(x)<\delta\})$;
\item $k$ is nondecreasing on $(0,\delta)$;
\item $1+(c-\varepsilon)(d^{\pm}(x))^{\theta}<b(x)/k^{2}(d^{\pm}(x))<1+(c+\varepsilon)(d^{\pm}(x))^{\theta},\ x\in\B_{\delta}$ (by~\ref{itm:b1},~\eqref{assum:b2});
\item $\phi'(t)<0,\ \phi''(t)>0$ for $t\in(0,\delta)$;
\item $\lim\limits_{r\to 0+}(I_1(r)+I_{2}^{\pm}(r)+ I_3^{\pm}(r)+I_4^{\pm}(r))=\pm\varepsilon$, for $r\in(0,\delta)$ (by~Lemma~\ref{lem:lims}).
\end{enumerate}
Thus with the free choice of $\delta>0$ small enough such that
\begin{align}
k^{2}(d^{+})(1+(c+\varepsilon)(d^{+})^{\theta})f(w^{+})-\Delta_{d^{\alpha},2}w^{+}\ge 0,\ x\in \B^{+}_{\rho};\label{eq:sup:constr2}\\
k^{2}(d^{-})(1+(c-\varepsilon)(d^{-})^{\theta})f(w^{-})-\Delta_{d^{\alpha},2}w^{-}\le 0,\ x\in \B^{-}_{\rho}.\label{eq:sub:constr2}
\end{align}
Now suppose that $u$ is a nonnegative solution of~\eqref{eq:gen:2}. First we note that for $M$ sufficiently large
\begin{align*}
u\le w^{-}+Mv(x),\quad\text{on}\ \partial\B_{\mu},\ \text{and}\quad 
w^{+}\le u+Mv(x),\quad\text{on}\ \partial\B_{\mu-\rho},
\end{align*}
where $v$ is the solution of the problem~\eqref{eq:aux}.
We observe that $w^{-}(x)\to\infty$ as $d^{-}(x)\to \rho$, and $u_{|_{\partial\B}}=+\infty>w^{+}(x)_{|_{\partial\B}}$. 
It follows from comparison principle~\cite[see Lemma 3.18]{heino12}, that
\begin{align*}
u\le w^{-}+Mv(x),\quad\text{on}\ \B_{\rho}^{-},\quad \text{and}\ 
w^{+}\le u+Mv(x),\quad\text{on}\ \B_{\rho}^{+}.
\end{align*}
Thus
\begin{align}
\chi^{-}\ge \left[ -1+\frac{u(x)}{\xi_0\phi(K(d^{-}(x)))}\right](y(d^{-}(x)))-\frac{Mv(x)}{\xi_0 y(d^{-}(x))\phi(K(d^{-}(x)))};\\
\chi^{+}\le \left[ -1+\frac{u(x)}{\xi_0\phi(K(d^{+}(x)))}\right](y(d^{+}(x)))+\frac{Mv(x)}{\xi_0 y(d^{+}(x))\phi(K(d^{+}(x)))}.
\end{align}
Hence, for $x\in\B^{-}_{\rho}\cap\B^{+}_{\rho}$, we have
\begin{align*}
\chi^{+}-\frac{Mv(x)}{\xi_0 y(d^{+}(x))\phi(K(d^{+}(x)))}\le \left[ -1+\frac{u(x)}{\xi_0\phi(K(d^{+}(x)))}\right](y(d^{+}(x)));~~~\\
\text{and}~~~~~~~~~~~~~\left[ -1+\frac{u(x)}{\xi_0\phi(K(d^{-}(x)))}\right](y(d^{-}(x))) \le \chi^{-}+\frac{Mv(x)}{\xi_0 y(d^{-}(x))\phi(K(d^{-}(x)))}.
\end{align*}
On letting $\rho\to 0$, we see that
\begin{align*}
\chi^{+}-\frac{Mv(x)}{\xi_0 y(d(x))\phi(K(d(x)))}\le \left[ -1+\frac{u(x)}{\xi_0\phi(K(d(x)))}\right]\le \chi^{-}+\frac{Mv(x)}{\xi_0 y(d(x))\phi(K(d(x)))}.
\end{align*}
On recalling that $\phi(t)\to \infty$ as $t\to 0$ and~\eqref{eq:aux:sol} along with assumption~\ref{itm:y},
 we have
\begin{align}\label{eq:growth2a}
\chi\le \liminf_{d(x)\to 0}\left[ -1+\frac{u(x)}{\xi_0\phi(K(d(x)))}\right]
\le \limsup_{d(x)\to 0}\left[ -1+\frac{u(x)}{\xi_0\phi(K(d(x)))}\right]
\le \chi.
\end{align}
The claimed result in~\eqref{eq:growth:2nd:a}-\eqref{eq:growth:2nd:b} follows by letting $\varepsilon\to 0$ in~\eqref{eq:growth2a}. Hence, it completes the proof.

\bibliographystyle{unsrt}
\bibliography{ref1}
\end{document}